\documentclass[11pt,a4paper,reqno]{amsart}
\usepackage{amsmath,amscd,amssymb,latexsym}
\usepackage{longtable}
\usepackage{xcolor}

\newcommand{\comment}[1]{} 
\def\abstand{$\vrule width 0pt height 13pt $}

\newtheorem{Thm}{Theorem}[section]
\newtheorem{Cor}[Thm]{Corollary}
\newtheorem{Lem}[Thm]{Lemma}
\newtheorem{Prop}[Thm]{Proposition}
\theoremstyle{definition}
\newtheorem{Def}[Thm]{Definition}
\newtheorem{Exm}[Thm]{Example}

\newtheorem{Rk}[Thm]{Remark}

\newcommand{\sm}{\left(\smallmatrix}
\newcommand{\esm}{\endsmallmatrix\right)}
\newcommand{\Z}{\mathbb Z}
\newcommand{\Q}{\mathbb Q}
\newcommand{\C}{\mathbb C}
\newcommand{\R}{\mathbb R}
\newcommand{\F}{\mathbb F}

\begin{document}
\title[Bielliptic intermediate modular curves]{Bielliptic intermediate modular curves}
\author[D. Jeon, C.-H. Kim and A. Schweizer]{Daeyeol Jeon, Chang Heon Kim and Andreas Schweizer}

\address{Daeyeol Jeon, 
Department of Mathematics Education, 
Kongju National University, 
182 Shinkwan-dong, Kongju, Chungnam 314-701, South Korea}
\email{dyjeon@kongju.ac.kr}

\address{Chang Heon Kim,
Department of Mathematics, 
Sungkyunkwan University, 
Suwon 440-746, South Korea}
\email{chhkim@skku.edu}

\address{Andreas Schweizer,
Department of Mathematics, 
Korea Advanced Institute of Science and Technology (KAIST), 
Daejeon 305-701, South Korea}
\email{schweizer@kaist.ac.kr}

\thanks{The first author was supported by the Basic Science Research Program through the National Research Foundation of Korea (NRF)
funded by the Ministry of Education
(NRF-2016R1D1A1B03934504).
The second author was supported by the
National Research Foundation of Korea (NRF) grant funded by the Korea government(MSIP) (NRF-2015R1D1A1A01057428 and 2016R1A5A1008055). The second and the third author were supported by 
the National Research Foundation of Korea (NRF) grant funded
by the Korean government (MSIP) (ASARC, NRF-2007-0056093).}

\begin{abstract}
We determine which of the modular curves $X_\Delta(N)$, that is, curves
lying between $X_0(N)$ and $X_1(N)$, are bielliptic. Somewhat surprisingly, 
we find that one of these curves has exceptional automorphisms. Finally 
we find all $X_\Delta(N)$ that have infinitely many quadratic points over 
$\mathbb{Q}$.
\end{abstract}

\maketitle \setcounter{section}{-1}

\section{Introduction}
Throughout this paper we assume that a curve $X$ is always defined over a field of characteristic $0$. 
For any positive integer $N$, let $\Gamma_1(N)$, $\Gamma_0(N)$ be the congruence subgroups
of $\Gamma={\rm SL}_2(\mathbb Z)$ consisting of the matrices $\sm a&b\\c&d\esm$ congruent
modulo $N$ to $\sm 1&*\\0&1\esm$, $\sm *&*\\0&*\esm$ respectively.
We let $X_1(N)$, $X_0(N)$ be the modular curves associated to
$\Gamma_1(N)$, $\Gamma_0(N)$ respectively.
Let $\Delta$ be a subgroup of $({\mathbb Z}/{N \mathbb Z})^*$, and let
$X_\Delta(N)$ be the modular curve defined over $\mathbb Q$
associated to the modular group $\Gamma_\Delta(N):$

$$\Gamma_\Delta(N)=\left\{{\begin{pmatrix}a&b\\c&d\end{pmatrix}}
\in\Gamma\,|\,c\equiv 0\,\, {\rm mod}\,\, N, (a\,\, {\rm mod}\,\, N)\in\Delta\right\}.$$

Since the negative of the unit matrix acts as identity on the complex upper halfplane we
have $X_\Delta(N)=X_{\langle \pm1,\Delta\rangle}(N)$. And since we are only interested in 
these curves (and not for example in modular forms for such a group $\Gamma_{\Delta}(N)$), 
we can and will always assume that $-1\in\Delta.$ 
\par 
If $\Delta=\{\pm 1\}$ (resp.
$\Delta=({\mathbb Z}/{N \mathbb Z})^*$) then $X_\Delta(N)$ is
equal to $X_1(N)$ (resp. $X_0(N)$). If $\Delta_1$ is a subgroup of $ \Delta_2$, the inclusions 
$\pm\Gamma_1(N)\subseteq\Gamma_{\Delta_1}(N)
\subseteq\Gamma_{\Delta_2}(N)\subseteq\Gamma_0(N)$
induce natural Galois covers 
$X_1(N)\to X_{\Delta_1}(N)\to X_{\Delta_2}(N) \to X_0(N)$. 
Denote the genus of $X_\Delta(N)$ by $g_\Delta(N)$.
\par
We point out that each such curve always has a $\Q$-rational point, namely the cusp $0$. 
See Lemma \ref{lem:cuspsfieldofdef} below. This simplifies the discussion 
of properties like being a rational curve, being elliptic,
or being hyperelliptic.
\par
The next interesting property after being hyperelliptic is being bielliptic. 
A curve $X$ of genus $g(X)\geq 2$ defined over a field $F$ is called {\it bielliptic} 
if it admits a map $\phi:X\to E$ of degree $2$ over an algebraic closure $\overline{F}$ onto an elliptic curve $E$.
Equivalently, $X$ has an involution $v$ (called {\it bielliptic involution}) such that
$X/\langle v\rangle\cong E$.
\par
Harris and Silverman \cite[Corollary 3]{H-S} showed that if a curve 
$X$ with $g(X)\geq 2$ defined over a number field $F$ is neither
hyperelliptic nor bielliptic, then the set of quadratic points on
$X,$ $$\{P\in X(\overline F):[F(P):F]\leq 2\}$$ is finite.
See \cite[Theorem 2.14]{B2} (cited as Theorem \ref{thm:precise} 
below) for a precise ``if-and-only-if''-statement.
\par
Bars \cite{B1} determined all the bielliptic modular curves of
type $X_0(N)$ and also found all the curves $X_0(N)$ which have
infinitely many quadratic points over $\mathbb Q$. The first two 
authors \cite{J-K1} did the same work for the curves $X_1(N)$.
\par
In this paper we shall determine all the 
bielliptic intermediate modular curves $X_\Delta(N)$ and
find all the curves $X_\Delta(N)$ which have
infinitely many quadratic points over $\mathbb Q$.
This type of curves is in some sense more difficult to handle 
than the other ones,
because there are various subgroups $\Delta$ and the 
automorphism groups of the curves $X_\Delta(N)$ are not determined 
yet for most cases.
\par

The first key tool is the following result.
\begin{Prop}\cite[Proposition 1]{H-S}\label{image}
If $X$ is a bielliptic curve, and if $X\to Y$ is a finite map, then $Y$ is subhyperelliptic (i.e. rational, elliptic or hyperelliptic) or bielliptic.
\end{Prop}

Thus it suffices to consider only those $N$ for which $X_0(N)$ is subhyperelliptic or bielliptic (or both).
Indeed, there are only finitely many such $N$ by Table \ref{classification} (see Appendix).
For these $N$ we list in Table \ref{lowgenus} all cases for which there is no intermediate $\Delta$
and all possible $\Delta$ for which $g_\Delta(N)\le 1$.
In Table \ref{highgenus} we list all $\Delta$ for these $N$ with $g_\Delta(N)\ge 2$ together with their genera, 
and we indicate in which part of the paper they are treated. 

Our main results are as follows:

\begin{Thm}\label{thm:bielliptic} The full list of bielliptic modular curves 
$X_\Delta(N)$ with $\{\pm 1\}\subsetneq\Delta\subsetneq (\Z/N\Z)^*$
is given by the following $25$ curves. Some of these curves have more bielliptic 
involutions than the ones we listed.

\begin{center}
\begin{longtable}{c|c|c}
\caption{List of all bielliptic $X_{\Delta}(N)$ and some of their bielliptic involutions}\\\abstand
$X_\Delta(N)$ & genus & some bielliptic involutions
 \\ \hline\abstand
 $X_{\Delta_1}(21)$ & $3$ & $\widehat{W}_{21}$,\ \ $[2]\widehat{W}_{21}$,\ \ 
$[4]\widehat{W}_{21}$
 \\ \hline\abstand
 $X_{\Delta_1}(24)$ & $3$ & $[7]$,\ \ $\widehat{W}_{24}$,\ \ 
$[7]\widehat{W}_{24}$,\ \ $\widehat{W}_8$,\ \ $[7]\widehat{W}_8$ 
 \\ \hline\abstand
 $X_{\Delta_2}(24)$ & $3$ & $[5]$,\ \ $\widehat{W}_{24}$,\ \ 
$[5]\widehat{W}_{24}$,\ \ $\widehat{W}_8$,\ \ $[5]\widehat{W}_8$ 
 \\ \hline\abstand
 $X_{\Delta_1}(26)$ & $4$ & $\widehat{W}_{26}$,\ \ $[3]\widehat{W}_{26}$,\ \ 
$[9]\widehat{W}_{26}$ 
 \\ \hline\abstand
 $X_{\Delta_2}(26)$ & $4$ & $\widehat{W}_{26}$, $[5]\widehat{W}_{26}$
 \\ \hline\abstand
 $X_{\Delta_1}(28)$ & $4$ & $\sm 1&0\\14&1\esm$,\ \ $\widehat{W}_{7}$,\ \ 
$[3]\widehat{W}_{7}$,\ \ $[9]\widehat{W}_{7}$ 
 \\ \hline\abstand
 $X_{\Delta_2}(28)$ & $4$ & $\widehat{W}_{7}$,\ \ $[11]\widehat{W}_{7}$
 \\ \hline\abstand
 $X_{\Delta_2}(29)$ & $4$ & $\widehat{W}_{29}$, $[2]\widehat{W}_{29}$
 \\ \hline\abstand
 $X_{\Delta_1}(30)$ & $5$ & $\widehat{W}_{15}$,\ \ $[7]\widehat{W}_{15}$
 \\ \hline\abstand
 $X_{\Delta_1}(32)$ & $5$ & $[7]$
 \\ \hline\abstand
 $X_{\Delta_2}(33)$ & $5$ & $\widehat{W}_{11}$,\ \ $[5]\widehat{W}_{11}$
 \\ \hline\abstand
 $X_{\Delta_2}(34)$ & $5$ & $\widehat{W}_2$
 \\ \hline\abstand
 $X_{\Delta_3}(35)$ & $7$ & $\widehat{W}_5$
 \\ \hline\abstand
 $X_{\Delta_4}(35)$ & $5$ & $\widehat{W}_{35}$,\ \ $[2]\widehat{W}_{35}$
 \\ \hline\abstand
 $X_{\Delta_2}(36)$ & $3$ & $[5]$,\ \ $\widehat{W}_{36}$,\ \ 
$[5]\widehat{W}_{36}$
 \\ \hline\abstand
 $X_{\Delta_3}(37)$ & $4$ & see Theorem \ref{thm:Aut37}
 \\ \hline\abstand
 $X_{\Delta_4}(39)$ & $5$ & $\widehat{W}_{39}$,\ \ $[2]\widehat{W}_{39}$
 \\ \hline\abstand
 $X_{\Delta_6}(40)$ & $5$ & $\sm 1&0\\20&1\esm$,\ \ $\sm -10&1\\-120&10\esm$,\ \ 
$[3]\sm -10&1\\-120&10\esm$
 \\ \hline\abstand
 $X_{\Delta_4}(41)$ & $5$ & $\widehat{W}_{41}$,\ \ $[3]\widehat{W}_{41}$
 \\ \hline\abstand
 $X_{\Delta_4}(45)$ & $5$ & $\widehat{W}_9$
 \\ \hline\abstand
 $X_{\Delta_6}(48)$ & $5$ & $\sm 1&0\\24&1\esm$,\ \ $\sm -6&1\\-48&6\esm$,\ \ 
$[5]\sm -6&1\\-48&6\esm$
 \\ \hline\abstand
 $X_{\Delta_2}(49)$ & $3$ & $\widehat{W}_{49}$, $[2]\widehat{W}_{49}$, $[4]\widehat{W}_{49}$
 \\ \hline\abstand
 $X_{\Delta_2}(50)$ & $4$ & $\widehat{W}_{50}$, $[3]\widehat{W}_{50}$
 \\ \hline\abstand
 $X_{\Delta_4}(55)$ & $9$ & $\widehat{W}_{11}$
 \\ \hline\abstand
 $X_{\Delta_3}(64)$ & $5$ & $\sm 1&0\\32&1\esm$
 \\ \hline
\end{longtable}
\end{center}
Throughout the paper the notation $X_{\Delta_i}(N)$ always refers to Tables \ref{lowgenus} and 
\ref{highgenus} in the Appendix, where the subgroups $\Delta_i$ are listed.
The notations for the bielliptic involutions are explained in Section ~\ref{sec:autos}.
\end{Thm}

\begin{Thm}\label{thm:bielliptic2} 
Let $\{\pm 1\}\subsetneq\Delta\subsetneq (\Z/N\Z)^*$.
Then the only modular curve $X_\Delta(N)$ of genus $>1$ which has infinitely many 
quadratic points over $\Q$ is the unique hyperelliptic curve $X_{\Delta_1}(21)$ 
where $\Delta_1 =\{\pm1,\pm 8\}$.
\end{Thm}

We say a few words about the moduli problem described by 
the curves $X_{\Delta}(N)$.
Let $\{\pm1\}\subseteq \Delta\subseteq (\Z/N\Z)^*$. A non-cuspidal
$K$-rational point of $X_{\Delta}(N)$ corresponds to an isomorphism 
class of an elliptic curve $E$ over $K$ (in short Weierstrass form
$y^2 =x^3 +Ax+B$) and a primitive $N$-torsion point $P$ of $E$ such 
that the set $\{aP\ :\ a\in\Delta\}$ is Galois stable under 
$Gal(\overline{K}/K)$.
\par
If $\Delta=\{\pm 1\}$, this means that the $x$-coordinate $x(P)$ of 
$P$ is in $K$. Consequently $(y(P))^2\in K$. We can take a quadratic 
twist of $E$ over $K$ that multiplies $x$ with $(y(P))^2$ and $y$
with $(y(P))^3$ and obtain an elliptic curve $E'$ over $K$ with
a $K$-rational $N$-torsion point.
\par
For general $\Delta$, since $Gal(\overline{K}/K)$ cannot map $P$ out
of $\Delta$ and since $P$ and $-P$ have the same $x$-coordinate, $x(P)$
has degree at most $\frac{1}{2}|\Delta|$ over $K$ (provided $N>2$). 
\par 
Conversely, if the degree of $x(P)$ over $K$ is even smaller than
$\frac{1}{2}|\Delta|$, then the set of all $n$ in $(\Z/N\Z)^*$ such 
that $nP$ is an image of $P$ under the action of $Gal(\overline{K}/K)$
and $P\mapsto -P$ forms a $Gal(\overline{K}/K)$-stable proper subgroup 
$\widetilde{\Delta}$ of $\Delta$, giving even rise to a $K$-rational 
point on the curve $X_{\widetilde{\Delta}}(N)$ that covers $X_{\Delta}(N)$.
\par 
Now let $L=K(x(P))$. As above,
we can twist $E$ over $L$ and get an elliptic curve $E'$ over $L$ 
with $L$-rational $N$-torsion point.
\par
But if $\Delta$ is bigger than $\{\pm 1\}$, in general we cannot 
get an elliptic curve $E''$ over $K$ with an $N$-torsion point
of degree $\leq \frac{1}{2}|\Delta|$ over $K$.

\begin{Exm}
The curve $X_{\Delta_1}(17)$ where $\Delta_1 =\{\pm 1,\pm 4\}$ 
has genus $1$ and hence infinitely many quadratic points. 
A non-cuspidal one corresponds to an elliptic curve $E$ over 
a quadratic number field $K$ with a $17$-torsion point $P$ 
whose $x$-coordinate has degree $2$ over $K$.
\par
Suppose we could construct from this an elliptic curve $E'$ over 
a quadratic number field $L$ with a $17$-torsion point $Q$ that
is quadratic over $L$. As the automorphism group of $(\Z/17\Z)^*$ 
is cyclic, the non-trivial Galois automorphism of $L(Q)/L$ can map
$Q$ only to $-Q$. But then $(E',Q)$ would correspond to a non-cuspidal
$L$-rational point on $X_1(17)$, which by \cite[Theorem 3.1]{Ka1} is 
known not to exist.
\par
Or interpreted differently, then we could twist $E'$ and get an
elliptic curve $E''$ over the quadratic number field $L$ with
an $L$-rational $17$-torsion point, which is known not to exist
(\cite[Theorem 3.1]{Ka1}).
\end{Exm}

\begin{Exm}
The curve $X_{\Delta_1}(21)$ with $\Delta_1 =\{\pm 1,\pm 8\}$ 
has infinitely many quadratic points over $\Q$. 
So there are infinitely many elliptic curves $E$ over quadratic number 
fields $K$ (depending on $E$) with a $K$-rational $21$-isogeny containing
a $K$-rational $7$-torsion point. Equivalently, we can say that there are
infinitely many elliptic curves $E$ over quadratic number fields $K$ that
have a $K$-rational $7$-torsion point and a $K$-rational $3$-isogeny.
\par
As by \cite[Theorem 2.1]{K-M} the curve $X_1(21)$ is known not to have any 
quadratic points outside the cusps, the underlying $3$-torsion point cannot 
be $K$-rational. Of course, taking a suitable twist of $E$ will make the 
$3$-torsion point $K$-rational, but then the ($y$-coordinate of the) 
$7$-torsion point will no longer be $K$-rational.
\par
Group-theoretically the underlying feature is that the intersection 
of $\pm\Gamma_1 (7)$ and $\pm\Gamma_1 (3)$ is $\Gamma_{\Delta_1}(21)$,
and not $\pm\Gamma_1 (21)$.
\end{Exm}

\begin{Exm}
By \cite[Theorem 4.3]{B1} the curve $X_0(61)$ has infinitely many 
quadratic points over $\Q$. This means that there are infinitely 
many elliptic curves $E$ over quadratic number fields $K$ (depending 
on $E$) with a $K$-rational $61$-isogeny. (Infinitely many here means 
with infinitely many different $j$-invariants.)
\par
However, none of the six intermediate curves $X_{\Delta}(61)$ has 
infinitely many quadratic points. So, as explained after Theorem \ref{thm:bielliptic2},
for almost all of these $j$-invariants the $x$-coordinate of the underlying $61$-torsion 
point of $E$ will generate an extension of $K$ of degree $30$.
\end{Exm}

\begin{Exm}
Fix a number field $F$. As the curve $X_0(41)$ is hyperelliptic, it 
has infinitely many quadratic points over $F$. So as $K$ varies over 
all quadratic extensions of $F$, there will be in total infinitely 
many elliptic curves $E$ over $K$ with a $K$-rational $41$-isogeny.
The $x$-coordinate of the underlying $41$-torsion point generates 
an extension of degree $1$, $2$, $4$, $5$, $10$ or $20$ of $K$,
depending on whether above our $K$-rational point on $X_0(41)$ there 
lies a $K$-rational point on $X_1(41)$, $X_{\Delta_i}(41)$ ($i=1,2,3,4$),
or not.  
\par 
For example, if $F$ contains $\Q(\sqrt{41})$ (then the involution 
$\widehat{W}_{41}$ of $X_{\Delta_4}(41)$ is defined over $F$ by Lemma 
\ref{Frickeinvolutions}) and if the elliptic curve 
$X_{\Delta_4}(41)/\widehat{W}_{41}$ has positive rank over $F$, then 
$X_{\Delta_4}(41)$ has infinitely many quadratic points over $F$, and 
for these there are elliptic curves with the $x$-coordinate of the 
$41$-torsion point lying already in a degree $10$ extension of $K$. 
\par
But no matter what $F$ is, there will always only be finitely many 
cases for which the $x$-coordinate of the $41$-torsion point generates 
an extension of $K$ of degree less than $10$. This is because none of 
the other intermediate curves $X_{\Delta}(41)$ is subhyperelliptic or 
bielliptic. So over any number field $F$ they will always have only 
finitely many quadratic points.
\end{Exm}

Recently \cite{Ze} described the normalizers of $\Gamma_{\Delta}(N)$ 
in $SL_2(\R)$ for several types of $\Delta$, 
which furnishes already quite a range of automorphisms
of $X_{\Delta}(N)$. Any automorphism that is not of this form is 
called {\it exceptional}. 
\par
The unpublished preprint \cite{Mo} states that if $N$ is 
square-free then $X_{\Delta}(N)$ has no exceptional automorphisms,
except for the well-known curve $X_0(37)$. However, deciding the
biellipticity of $X_{\Delta_3}(37)$ where 
$\Delta_3 =\{\pm 1,\pm 6,\pm 8,\pm 10,\pm 11,\pm 14\}$ revealed
that this curve also has exceptional automorphisms 
(see Lemma \ref{lem:37biell} and Theorem \ref{thm:Aut37} below). 
\par
As exactly the curve $X_{\Delta_3}(37)$ requires lengthy extra 
treatment in \cite{Mo}, it is quite likely that the mistake
occurs there and that the rest of \cite{Mo} is probably correct.
But checking this would require more algebraic geometry than we 
are comfortable with. 
\par
It might seem unfair that we elaborate on this mistake. After
all, \cite{Mo} is an unpublished preprint. But some of the proofs 
in the paper \cite{I-M} by Ishii and Momose use results from 
\cite{Mo}; so their correctness is an issue. And on the other
hand, after eliminating the mistake in \cite{Mo} one would have
a result that is much more general than the currently established 
ones. The paper \cite{Ka2} determines the automorphism group of
$X_\Delta(N)$ for $N$ a prime bigger than $311$ (which is too big 
to be of any help to us). Both, \cite{Ka2} itself and its 
review in MathSciNet suggest that it should be possible to
generalize to the case of square-free $N$.
\par
We want to emphasize that our results in this paper do not depend on \cite{Mo} or \cite{I-M}. 
This paper is organized as follows. 
\par
In Section~\ref{sec:autos} we have a more detailed look at two types
of automorphisms of $X_{\Delta}(N)$, namely those from the Galois group
of $X_{\Delta}(N)$ over $X_0(N)$ and, more interestingly, the possible
lifts of the Atkin-Lehner involutions to $X_{\Delta}(N)$. These can 
behave quite differently from the Atkin-Lehner involutions on $X_0(N)$.
\par
Section~\ref{sec:Atkin} presents a systematic method to find all the fixed points on $X_0(N)$ (resp.  
$X_\Delta(N)$) by Atkin-Lehner involutions (resp. possible lifts of Atkin-Lehner involutions).
Formulas for the number of fixed points on $X_0(N)$ were already known to Newman and Ogg.
Also Delaunay \cite{De} suggested a method to find all the fixed points.
In fact, he gave an algorithm to give all the candidates for the fixed points, 
but didn't explain how to choose the exact fixed points among them explicitly.
\par
In Section~\ref{sec:non} we exclude all the non-bielliptic curves 
$X_\Delta(N)$ by using various criteria, and in Section~\ref{sec:bielliptic} 
we show that the remaining $X_\Delta(N)$ are bielliptic curves,
so that Theorem \ref{thm:bielliptic} is proved.
\par
Lastly, in Section~\ref{sec:quadratic} we prove Theorem \ref{thm:bielliptic2}, i.e., we find all the curves 
$X_\Delta(N)$ which have infinitely many quadratic points over $\mathbb Q$.


\section{Automorphisms}\label{sec:autos}

In this section we describe some automorphisms of $X_\Delta(N)$.
We mainly need two types of non-exceptional automorphisms, that
are easy to handle.
\par
Note that $X_\Delta(N)\to X_0(N)$ is a Galois covering with Galois
group $\Gamma_0(N)/\Gamma_\Delta(N)$ which is isomorphic to
$(\Z/N\Z)^*/\Delta$. 
For an integer $a$ prime to $N,$ let $[a]$ denote the automorphism
of $X_\Delta(N)$ represented by $\gamma\in\Gamma_0(N)$ such that
$\gamma\equiv\sm a&*\\0&*\esm\mod N.$ Sometimes we regard $[a]$ as
a matrix.
\par 
Before we come to the second type of automorphisms, we first say something
about the cusps of $X_\Delta(N)$, which will be used several times in the paper.
\par 
Let $X(N)$ be the modular curve defined over $\mathbb Q$
associated to the modular group $\Gamma(N):$
$$\Gamma(N)=\left\{{\begin{pmatrix}a&b\\c&d\end{pmatrix}}
\in\Gamma\,|\,a\equiv d\equiv 1\,\,{\rm mod}\,\, N, b\equiv c\equiv 0\,\,{\rm mod}\,\, N,\right\}.$$
By virtue of \cite{O1} the cusps of $X(N)$ can be regarded as pairs $\pm\sm
x\\y\esm$, where $x,y\in {\mathbb Z}/{N\mathbb Z},$ and are
relatively prime, and $\sm x\\y\esm$, $\sm -x\\-y\esm$ are
identified; $\Gamma_\Delta/{\Gamma(N)}$ operates naturally on the left,
and so a cusp of $X_\Delta(N)$ can be regarded as an orbit
of $\Gamma_\Delta(N)/{\Gamma(N)}$.
They are all defined over $\Q(\zeta_N)$ or subfields thereof and 
$Gal(\Q(\zeta_N)/\Q)=(\Z/N\Z)^*$ acts on them by 
$\sigma {x\choose y}={\sigma x\choose y}$. 

\begin{Lem}\cite[Lemma 1.2]{I-M}\label{lem:cuspsfieldofdef}
\begin{itemize}
	\item[(a)] Let $d$ be a divisor of $N$ such that $d$ and $N/d$ are
	relatively prime. Set
	$$\Delta^{(d)}=\{a\ mod\ d\ :\ a\in\Delta,\ a\equiv 1\ mod\ N/d\}.$$
	Then the field of definition of the cusps of $X_\Delta(N)$ lying above
	the cusp ${i\choose d}$ of $X_0(N)$ is $\Q(\zeta_d)^{\Delta^{(d)}}$, the 
	fixed field of $\Delta^{(d)}\subseteq (\Z/d\Z)^* =Gal(\Q(\zeta_d)/\Q)$
	in $\Q(\zeta_d)$.
	\item[(b)] The cusps of $X_\Delta(N)$ above the cusp $0$ ($={0\choose 1}={1\choose 1}$)
	of $X_0(N)$ are always $\Q$-rational.
	\item[(c)] The field of definition of the cusps of $X_\Delta(N)$ above the cusp 
	$\infty$ ($={1\choose 0}={1\choose N}$) of $X_0(N)$ is $\Q(\zeta_N)^{\Delta}$. 
	In particular, they are never rational over $\Q$ (if $\Delta\subsetneq(\Z/N\Z)^*$).
\end{itemize}
\end{Lem}

\begin{proof}
(a) \cite[Lemma 1.2]{I-M}. \\
(b), (c) special cases of (a) or directly from \cite[Proposition 1]{O1}.	
\end{proof}

\cite[Lemma 1.2]{I-M} also describes the field of definition of the cusps above
${i\choose d}$ if $d$ and $N/d$ are not coprime. This is slightly more complicated 
and will not be needed in our paper.

\begin{Rk}
Note that although $\Delta$ contains $\pm 1$, the field of definition of the cusp 
${1\choose d}$ need not be contained in the real subfield of $\Q(\zeta_d)$.
For example the cusp ${1\choose 3}$ of the curve $X_1(15)$ is defined over $\Q(\zeta_3)$,
because $\Delta=\{\pm 1\}$ and $\Delta^{(3)}=\{1\}$. 
\end{Rk}

For each divisor $d|N$ with $(d,N/d)=1,$ consider the matrices of
the form $\begin{pmatrix} dx & y\\Nz & dw \end{pmatrix}$ with
$x,y,z,w\in\mathbb Z$ and determinant $d.$ Then these matrices
define a unique involution of $X_0(N)$, which is called the 
{\it Atkin-Lehner involution} and denoted by $W_d$. In particular, 
if $d=N,$ then $W_N$ is called the {\it Fricke involution.}
We also denote by $W_d$ a matrix of the above form.

If we fix a matrix $W_d$ then $W_d$ may not belong to the
normalizer of $\Gamma_\Delta(N)$ in ${\rm PSL}_2(\mathbb R)$ and might therefore not define an automorphism of $X_\Delta(N)$.
Now we will find a criterion for which $W_d$ defines an automorphism on $X_\Delta(N)$.
Each $\gamma\in\Gamma_\Delta(N)$ is of the form
$$\begin{pmatrix} a & b\\c & \overline{a} \end{pmatrix}$$
where $(a\,\, {\rm mod}\,\, N)\in\Delta$ and $\overline{a}$ is an integer with $a\overline{a}\equiv1\,\,({\rm mod}\,N)$.
For $W_d=\begin{pmatrix} dx & y\\Nz & dw \end{pmatrix}$ and
$\gamma=\begin{pmatrix} a & b\\c & \overline{a} \end{pmatrix}\in\Gamma_\Delta(N)$,
one can easily compute that $W_d\gamma W_d^{-1}\in\Gamma_\Delta(N)$ 
if and only if the following condition holds:
\begin{equation}\label{eq:nor}
dxwa-\frac{N}{d}yz\overline{a}\in\Delta.
\end{equation}
Since $d^2xw-Nyz=d$, we have $dxw-\frac{N}{d}yz=1$, and hence the following 
holds:
\begin{equation*}\label{eq:condition}
dxwa-\frac{N}{d}yz\overline{a}\equiv\left\{\begin{array}{l}a\,\,({\rm mod}\,\frac{N}{d}),\\
\overline{a}\,\,({\rm mod}\,d).\end{array}\right.
\end{equation*}
Note that $\overline{a}$ is the multiplicative inverse of $a$ modulo $d$.
Now we define an isomorphism $t_d:(\Z/N\Z)^*\to(\Z/N\Z)^*$ by
$$t_d(a)\equiv\left\{\begin{array}{l}a\,\,({\rm mod}\,\frac{N}{d}),\\
\overline{a}\,\,({\rm mod}\,d).\end{array}\right.$$
Since $(\Z/N\Z)^*$ is isomorphic to the direct product $(\Z/d\Z)^*\times(\Z/\frac{N}{d}\Z)^*$, one can show that
the conditon (\ref{eq:nor}) holds if and only if $t_d(a)\in\Delta$.
Therefore we have the following result:

\begin{Lem}\label{thm:whenaut}
A matrix $W_d$ defines an automorphism of $X_\Delta(N)$ if and only if
$t_d(\Delta)=\Delta$.
\end{Lem} 

\begin{Rk} If $W_d$ defines an automorphism of $X_\Delta(N)$ then so does 
$W_{\frac{N}{d}}$ because $t_{\frac{N}{d}}=t_d^{-1}$. 
\par
Moreover, if $a$ is running through $(\Z/N\Z)^*/\Delta$, then $[a]W_d$
gives the different automorphisms of $X_{\Delta}(N)$ that induce the same
Atkin-Lehner involution $W_d$ on $X_0(N)$.
\end{Rk}

\begin{Exm}\label{exm:Wnotaut}
Let $N=65$. Then $t_5$ interchanges $\Delta_1 =\{\pm 1, \pm 8\}$ and 
$\Delta_3 =\{\pm 1, \pm 18\}$. Correspondingly, matrices $W_5$ do not give 
automorphisms of $X_{\Delta_1}(65)$; rather they give isomorphisms from
$X_{\Delta_1}(65)$ to $X_{\Delta_3}(65)$.
\end{Exm}

\begin{Exm}\label{exm:Wnotinvol}
Even for square-free $N$ it can happen that matrices $W_d$ give automorphisms
of $X_{\Delta}(N)$, but none of these automorphisms is an involution.
\par 
Take for example $N=65$ and $\Delta_2 =\{\pm 1, \pm 14\}$. Then any matrix
$W_5=\begin{pmatrix} 5x & y\\65z & 5w \end{pmatrix}$ with determinant 
$5$ gives an automorphism of $X_{\Delta_2}(N)$ by Lemma \ref{thm:whenaut}. 
However, $\frac{1}{5}W_5^2$ is a matrix in $\Gamma_0(65)$ with upper 
left entry $5x^2 +13yz$, which cannot be congruent to $\pm 1$ or 
$\pm 14$ modulo $65$, as $\pm 5$ is not a square modulo $13$. 
So $W_5^2$ gives a non-trivial automorphism of $X_{\Delta_2}(65)$.
\par
By the way, the same then holds as well for the automorphisms 
$W_5$ of $X_1(65)$.
\end{Exm}

Now we derive some conditions under which an automorphism of
$X_{\Delta}(N)$ coming from a matrix $W_d$  necessarily is an involution. 
\par
Assume that such automorphism has a non-cuspidal fixed point on $X_{\Delta}(N)$.
Then we can multiply $W_d$ on the left by a matrix from 
$\Gamma_{\Delta}(N)$ without changing the automorphism such that the 
new matrix actually has a fixed point on the upper half plane 
$\mathbb H$.
Note that one can find conditions for an involution $W_d$ on $X_0(N)$ to have non-cuspidal 
fixed points in \cite [p.454]{O2} and \cite[Proposition 3.8]{I-J-K}. Because of the cuspidal 
fixed point of $W_4$, statement (1) of \cite[Proposition 3.8]{I-J-K} should however be corrected
to "$W_Q$ has a non-cuspidal fixed point on $X_0(N)$".
\par
We recall a notion which plays a crucial role in this paper.
\begin{Def}\cite{Sh} A matrix $A\in {\rm GL}_2^+(\R)$ is called an {\it elliptic element} 
if it is conjugate to
$$\begin{pmatrix} \lambda & 0\\ 0 & \bar\lambda \end{pmatrix},\,\, \lambda \notin \R.$$
\end{Def}
It is well-known that an elliptic element $A$ has a fixed point on $\mathbb H$, which is equivalent to 
$|{\rm tr}(A)|^2<4\det(A)$ \cite[Propositions 1.12 and 1.13]{Sh}.
\par
So if $d>3$ and a matrix $W_d$ has a fixed point on $\mathbb H$, this forces
${\rm tr}(W_d)=0$. In particular, then $W_d$ defines an involution on $X_\Delta(N)$.
Moreover, if ${\rm tr}(W_d)=0$, from $\det(W_d)=d$ we get
$-dx^2 -\frac{N}{d}yz=1$; so $-d$ necessarily is congruent 
to a square modulo $\frac{N}{d}$.

Conversely, suppose the following conditions hold:\\ \\
for $d\neq 2,3$:
\begin{enumerate}
\item[(i)] $t_d(\Delta)=\Delta$ and
\item[(ii)] $-d$ is congruent to a square modulo $\frac{N}{d}$.
\end{enumerate}
\bigskip
\noindent 
for $d=2$ or $3$:
\begin{enumerate}
\item[(i)] $t_d(\Delta)=\Delta$ and
\item[(ii)$'$] $dx(t-x) \equiv 1\,\, ({\rm mod}\,\, N/d)$
has a solution for some $t\in \{0,1,-1\}$.
\end{enumerate}

\bigskip 
\noindent 
First consider the case $d\neq 2,3$.
Then there exist $x,y$ satisfying $-dx^2-y\frac{N}{d}=1$, and for uniqueness if we set $x_0$ to be the smallest nonnegative such integer and $y_0=\frac{-d x_0^2 -1}{N/d}$ then the matrix
\begin{equation}\label{eq:wd}
\widehat{W}_d=\begin{pmatrix}dx_0&y_0\\N&-dx_0\end{pmatrix}
\end{equation}
defines an involution on $X_\Delta(N)$, and it has a non-cuspidal
fixed point on $X_\Delta(N)$.
Note that the condition (ii) is equivalent to the existence of an elliptic 
element $W_d$.

For the case $d=2$ or $3$, we encounter some different situation.
One can check easily that the condition (ii)$'$ is equivalent to the existence of an elliptic element $W_d$.
But in this case, such $W_d$ may not define an involution on $X_\Delta(N)$.
We choose $\widehat{W}_d$ a matrix of the form \eqref{eq:wd} if it exists, otherwise we can choose $\widehat{W}_d$ uniquely of the form
\begin{equation}\label{eq:wd2}
\widehat{W}_d=\begin{pmatrix}dx_0&y_0\\N&d(t-x_0)\end{pmatrix}
\end{equation}
with the smallest nonnegative integer $x_0$ satisfying (ii)$'$. 

In the above, we use the notation $\widehat{W}_d$ to distinguish it from the Atkin-Lehner involution on $X_0(N)$. Moreover $\widehat{W}_d$ means the matrix of the above form \eqref{eq:wd} or \eqref{eq:wd2}, or the automorphism on $X_\Delta(N)$ defined by such a matrix. 
\par 

In particular, we always have 
$\widehat{W}_N=\begin{pmatrix} 0 & -1\\N & 0 \end{pmatrix}$. 

\begin{Lem}\label{Frickeinvolutions}
The matrix 
$\widehat{W}_N=\begin{pmatrix} 0 & -1\\N & 0 \end{pmatrix}$
normalizes the Galois group of $X_{\Delta}(N)$ over $X_0(N)$ with 
$\widehat{W}_N[a]\widehat{W}_N^{-1}=[\overline{a}]$.
So each automorphism $[a]\widehat{W}_N$ of $X_{\Delta}(N)$ is an involution.
\par 
Moreover, the field of definition for any automorphism $[a]\widehat{W}_N$ of 
$X_{\Delta}(N)$ is the fixed field $\Q(\zeta_N)^{\Delta}$ of 
$\Delta\subseteq (\Z/N\Z)^* =Gal(\Q(\zeta_N)/\Q)$ in $\Q(\zeta_N)$.
In particular, $[a]\widehat{W}_N$ is never defined over $\Q$ 
(when $\Delta\subsetneq (\Z/N\Z)^*$).
\end{Lem} 

\begin{proof}
The first few claims are easily seen by direct calculation. 
Now $[a]\widehat{W}_N$ maps the ($\Q$-rational) cusps above the cusp 
${0\choose 1}$ of $X_0(N)$ to cusps above $\infty ={1\choose 0}={1\choose N}$.
By Lemma \ref{lem:cuspsfieldofdef} the latter ones are only defined over 
$\Q(\zeta_N)^{\Delta}$. So the field of definition of $[a]\widehat{W}_N$ must
contain $\Q(\zeta_N)^{\Delta}$. On the other hand, it cannot be bigger, because
there are only $|(\Z/N\Z)^*/\Delta|$ automorphisms of $X_{\Delta}(N)$ lying
above the automorphism $W_N$ of $X_0(N)$.  		
\end{proof}

For example if $\Delta_4=\{\pm 1,\pm 3,\pm 4,\pm 7,\pm 9,\pm 10,\pm 11,\pm 12,\pm 16\}$,
then the involution $\widehat{W}_{37}$ of $X_{\Delta_4}(37)$ is only defined over 
$\Q(\sqrt{37})$.
This will turn out to be useful in the proof of Proposition \ref{prop:37noex}. 
By the way, this also explains that the genus $2$ curve 
$X_{\Delta_4}(37)/\widehat{W}_{37}$ does not show up in the tables of \cite{B-G-G-P}.

\begin{Exm}
However, if $3<d\neq N$, the existence of a matrix $\widehat{W}_d$ does 
not imply that all automorphisms $W_d$ on $X_\Delta(N)$ are involutions. Let for example 
$N=55$ and $\Delta_3 =\{\pm 1, \pm 16, \pm 19, \pm 24, \pm 26\}$.
Then $t_{11}(\Delta_3)=\Delta_3$ and a matrix $\widehat{W}_{11}$ exists
because $-11$ is a square modulo $5$. But the automorphism $W_{11}$ on $X_{\Delta_3}(55)$ defined by $\begin{pmatrix} 11 & 2\\55 & 11 \end{pmatrix}$
has order $4$ because its square is $[21]$.
\end{Exm}

\begin{Exm}\label{exm:dontcommute}
Even if $\widehat{W}_{d_1}$ and $\widehat{W}_{d_2}$ are involutions 
of $X_{\Delta}(N)$, their product is not necessarily an involution.
In particular, such involutions on $X_{\Delta}(N)$ do not always commute.
\par For example, $\widehat{W}_5 ={10\ -3\choose 35\ -10}$ and 
$\widehat{W}_{35}={0\ -1\choose 35\ 0}$ define involutions on 
$X_{\Delta_2}(35)$ where $\Delta_2 =\{ \pm 1, \pm 11, \pm 16 \}$.
But $(\widehat{W}_5 \widehat{W}_{35})^2=[13]$ has order $4$; 
so $\widehat{W}_5 \widehat{W}_{35}$ has order $8$.
\end{Exm}




\section{Fixed points of Atkin-Lehner involutions}\label{sec:Atkin}

In this section we explain a method to determine whether a matrix 
$W_d$ gives a bielliptic involution on $X_\Delta(N)$ or not.
Thanks to the Hurwitz formula, it suffices to know how to calculate 
the number of the fixed points of $W_d$. 
\par
First we point out that if $d\neq 4$, then $W_d$ does not fix any 
cusps by \cite[Proposition 3]{O2}. For our purposes it will suffice 
to find the number of non-cuspidal fixed points of $W_d$ on 
$X_{\Delta}(N)$.
\par
Delaunay \cite{De} suggested 
a method to find all the fixed points of $W_d$ on $X_0(N)$. In fact, 
he gave an algorithm to give all the candidates for the fixed points 
of $W_d$, but didn't explain how to choose the exact fixed points 
among them explicitly.

Suppose $d\neq 2,3$. 
If $W_d$ has a non-cuspidal fixed point on $X_0(N)$ at all 
(and only these $W_d$ are of interest to us), then $W_d$ is given 
by an elliptic element (see Section \ref{sec:autos}). So $x=-w$ and
$$W_d=\begin{pmatrix} dx & y\\ Nz & -dx\end{pmatrix}.$$
One can check easily that
\begin{equation}\label{eq:fixedpoint}
\tau=\frac{2dx+\sqrt{-4d}}{2Nz}
\end{equation}
is a fixed point of $W_d$. Conversely, every fixed point has the form (\ref{eq:fixedpoint}).

Suppose $d=2$ or $3$. 
If $W_d$ has a non-cuspidal fixed point at all, then, as explained in the 
last section, we have $|x+w|=0$ or $|x+w|=1$. The first case 
is discussed above, for the second, putting $w=1-x$ the point
\begin{equation}\label{eq:fixedpoint2}
\tau=\frac{-d(1-2x)+\sqrt{d^2-4d}}{2Nz}.
\end{equation}
is fixed by
$$W_d=\begin{pmatrix} dx & y\\ Nz & d(1-x)\end{pmatrix}.$$



Now we will propose a systematic way to find inequivalent points modulo 
$\Gamma_0(N)$ among the fixed points in (\ref{eq:fixedpoint}) and 
(\ref{eq:fixedpoint2}). For this purpose we need to introduce quadratic 
forms. For a negative integer $D$ conguent to $0$ or $1$ modulo $4$, we 
denote by $\mathcal Q_D$
the set of positive definite integral binary quadratic forms
$$Q(x,y)=[p,q,r]=px^2+qxy+ry^2$$
with discriminant $D=q^2-4pr$.
Then $\Gamma(1)$ acts on $\mathcal Q_D$ by
$$Q\circ\gamma(x,y)=Q(sx+ty,ux+vy)$$
where $\gamma=\begin{pmatrix}s&t\\u&v\end{pmatrix}$. A primitive
positive definite form $[p,q,r]$ is said to be {\it reduced form}
if
$$|q|\leq p\leq r,\,\,\mbox{and}\,\,q\geq 0\,\,\mbox{if either}\,\,|q|=p\,\,\mbox{or}\,\,p=r.$$
Let $\mathcal Q_D^\circ\subset\mathcal Q_D$ be the subset of primitive forms, that is, 
$$
\mathcal Q_D^\circ:=\{[p,q,r] \in \mathcal Q_D \mid \gcd (p,q,r)=1 \}.
$$
Then $\Gamma(1)$ also acts on $\mathcal Q_D^\circ$.
As is well known \cite{C}, there is an 1-1 correspondence between the set of classes $\Gamma(1)\backslash\mathcal Q_D^\circ$ and the set of reduced primitive definite forms.


\begin{Prop}\cite{G-K-Z}\label{GKZ} For each $\beta \in \mathbb{Z}/2N\mathbb{Z}$, we define
$$\mathcal Q_{D,N,\beta}^\circ=\{[pN,q,r]\in\mathcal Q_D\,|\,  \beta\equiv q\,\,({\rm mod}\,\,2N), \gcd(p,q,r)=1\}. $$
Then we have the following:
\begin{enumerate}
\item[(i)] Define $m=\gcd\left(N,\beta,\frac{\beta^2-D}{4N}\right)$ and
fix a decomposition $m=m_1m_2$ with $m_1,m_2>0$ and $\gcd(m_1,m_2)=1.$
Let
$$\mathcal Q_{D,N,\beta,m_1,m_2}^\circ=\left\{[pN,q,r]\in\mathcal Q_{D,N,\beta}^\circ
\,|\,\gcd(N,q,p)=m_1,\gcd(N,q,r)=m_2\right\}.$$
Then $\Gamma_0(N)$ acts on $\mathcal Q_{D,N,\beta,m_1,m_2}^\circ$ and
there is an 1-1 correspondence between
\begin{eqnarray*}
\mathcal Q_{D,N,\beta,m_1,m_2}^\circ/\Gamma_0(N) & \to & 
\mathcal Q_D^\circ/\Gamma(1) \\ {[pN,q,r]} & \mapsto & [pN_1,q,rN_2]
\end{eqnarray*}
where $N_1 N_2$ is any decomposition of $N$
into coprime factors such that $\gcd(m_1,N_2)=\gcd(m_2,N_1)=1$. 
Moreover we have a $\Gamma_0(N)$-invariant decomposition as follows:
\begin{equation}\label{decom2}
\mathcal Q_{D,N,\beta}^\circ=  \bigcup_{m=m_1m_2\atop{m_1,m_2>0\atop{\gcd(m_1,m_2)=1}}}   \mathcal Q_{D,N,\beta,m_1,m_2}^\circ.
\end{equation}

\item[(ii)] The inverse image $[pN_2,q,r/N_2]$ of any primitive form $[\bar{p},\bar{q},\bar{r}]$ of discriminant $D$ under the 1-1 correspondence in (i) is obtained by solving the following equations:
\begin{eqnarray*}
&&p=\bar{p}s^2+\bar{q}su+\bar{r}u^2\\
&&q=2\bar{p}st+\bar{q}(ru+tu)+2\bar{r}uv\\
&&r=\bar{p}t^2+\bar{q}tv+\bar{r}v^2
\end{eqnarray*}
satisfying $p\equiv 0\,({\rm mod}\,\,N_1),q\equiv \beta\,({\rm mod}\,\,2N),
r\equiv 0\,({\rm mod}\,\,N_2)$ and $\begin{pmatrix}s&t\\u&v\end{pmatrix}\in\Gamma(1)$.

\item[(iii)] We have the following $\Gamma_0(N)$-invariant decompostion:
\begin{equation}\label{decom}
\mathcal Q_{D,N,\beta}=\bigcup_{{\ell>0}\atop{\ell^2 |D}}
\bigcup_{{\lambda\,(2N)}\atop{\ell\lambda\equiv\beta\,(2N)\atop \lambda^2\equiv{D/\ell^2\,(4N)}}}
\ell\mathcal Q^\circ_{D/\ell^2,N,\lambda}
\end{equation}
\end{enumerate}
\end{Prop}
Given a quadratic form $Q(x,y)$, let $\alpha_Q$ be the unique zero in $\mathbb H$ satisfying $Q(x,1)=0$.
Note that each fixed point in (\ref{eq:fixedpoint}) (resp. (\ref{eq:fixedpoint2})) can be considered as the
point $\alpha_Q$ of a quadratic form $Q=[Nz,-2dx,-y]$ (resp. $Q=[Nz,d(1-2x),-y]$).
Our strategy is to find inequivalent quadratic forms $[Nz,-2dx,-y]$ mod $\Gamma_0(N)$
by using Proposition \ref{GKZ} and the decomposition (\ref{decom}).
One can show that for $d=3$ the points in (\ref{eq:fixedpoint2}) are the same as
the points in (\ref{eq:fixedpoint}) of the quadratic forms in
$2\mathcal Q_{-3,N,\lambda}^\circ$. Thus it suffices to handle the case $d=2$ separately.

Regarding the computation of fixed points on $X_0(N)$, we can propose the following four steps.
\begin{enumerate}
\item[{\bf Step I}] We search $\beta \pmod{2N}$ such that $\beta^2\equiv -4d \pmod{4N}$ with $\beta \equiv -2dx\pmod{2N}$ (resp. $\beta \equiv d(1-2x) \pmod{2N}$) for $d>3$(resp. $d=2,3$)  where $x\in \mathbb Z$.

\item[{\bf Step II}] We set the decomposition as in \eqref{decom2} and \eqref{decom} with $D=-4d$ (resp. $D=d^2-4d$) for $d>3$ (resp. $d=2,3$).

\item[{\bf Step III}] For each factor in the decompostion in Step II, we find the quadratic form representations
and taking the inverse inverse image of reduced forms under the map which is described  in  Proposition \ref{GKZ}-(ii).

\item[{\bf Step IV}] We form the elliptic elements corresponding to quadratic form representations obtained in Step III,
and find their fixed points.
\end{enumerate}

Finally we explain a method of determining the number of fixed points of an elliptic element $[a]\widehat{W}_d$  on $X_\Delta(N)$ for some $a\in(\mathbb Z/N\mathbb Z)^*$.
Note that the fixed points of $[a]\widehat{W}_d$ on $X_\Delta(N)$ are lying above the fixed points of $W_d$ on $X_0(N)$.
Let $z_1,z_2,\dots,z_n\in\mathbb H$ be the inequivalent elliptic points which represent all the fixed points of $W_d$ on $X_0(N)$, and let  $W_{d,1},W_{d,2},\dots,W_{d,n}$
be corresponding elliptic elements.
Note that $G=\Gamma_0(N)/\Gamma_\Delta(N)$ is the Galois group of the covering $X_\Delta(N)\to X_0(N)$, and
the elements of $G$ are the automorphisms of the form $[g]$ with $g\in(\mathbb Z/N\mathbb Z)^*$.
Thus for each $j$ the points on $X_{\Delta}(N)$ lying above $z_j$ are represented by $[g]z_j$ with $[g]\in G$.
Then one can easily show that $W_{d,i}$ fixes $[g]z_j$ if and only if $W_{d,i}[g]{W_{d,j}}^{-1}[g]^{-1}\in\Gamma_\Delta(N)$.
Thus one can calculate the number of fixed points of $W_{d,i}$ lying above $z_j$, and hence determine whether $W_{d,i}$ defines a bielliptic involution or not on $X_\Delta(N)$.

\begin{Exm}\label{calcfixedpoints}
Consider $X_{\Delta_2}(34)$ of genus 5 where $\Delta_2 =\{\pm1,\pm9,\pm13,\pm15\}$.
Since $W_{2}$ a bielliptic involution of $X_0(34)$ of genus 3, it has $4$ fixed points on $X_0(34)$.
First, let us find $4$ elliptic elements which give $4$ fixed points of $W_2$ on $X_0(34)$ by following the algorithm proposed as above.
Let $N=34$ and $d=2$. Applying {\bf Step I} and  {\bf Step II} we have $D=-4d=-8$, $\beta\equiv \pm 20 \pmod{68}$, and have decomposition as follows:
$$
\mathcal{Q}_{-8, 34, 20}=\mathcal{Q}_{-8,34,20,1,1}^\circ, \mathcal{Q}_{-8, 34, -20}=\mathcal{Q}_{-8,34,-20,1,1}^\circ.$$
Also, we have $D=d^2-4d=-4$, $\beta\equiv \pm 26 \pmod{68}$, and have decomposition as follows:
$$
\mathcal{Q}_{-4, 34, 26}=\mathcal{Q}_{-4,34,26,1,1}^\circ, \mathcal{Q}_{-4, 34, -26}=\mathcal{Q}_{-4,34,-26,1,1}^\circ.$$
Now applying {\bf Step III} we obtain the following quadratic form representations:
\begin{align*}
\mathcal{Q}_{-8,34,20,1,1}^\circ /\Gamma_0(34) & = \{ [34,20,3]\},\\
\mathcal{Q}_{-8,34,-20,1,1}^\circ /\Gamma_0(34) & = \{ [34,-20,3]\},\\
\mathcal{Q}_{-4,34,26,1,1}^\circ /\Gamma_0(34) & = \{ [34,26,5]\},\\
\mathcal{Q}_{-4,34,-26,1,1}^\circ /\Gamma_0(34) & = \{ [34,-26,5]\}
\end{align*}
Moreover in {\bf Step IV} the corresponding elliptic elements are given as follows:
$$
W_{2,1}=\begin{pmatrix}-10&-3\\34&10\end{pmatrix},
W_{2,2}=\begin{pmatrix}10&-3\\34&-10\end{pmatrix},
W_{2,3}=\begin{pmatrix}-12&-5\\34&14\end{pmatrix},
W_{2,4}=\begin{pmatrix}14&-5\\34&-12\end{pmatrix},
$$
Then, on $X_{\Delta_2}(34)$, they define a unique involution $\widehat{W}_{2}$.
\par 
Applying the Hurwitz formula to the covering $X_{\Delta_2}(34)\to X_0(34)$, we see that
it is totally unramified; so there are $8$ points of $X_{\Delta_2}(34)$ lying above the
$4$ fixed points on $X_0(34)$ by $W_{2,i}$.
\par 
Now we show that $\widehat{W}_{2}$ fixes all these $8$ points. 
Let $z_i$ be the fixed points on $\mathbb H$ by $W_{2,i}$ for $i=1,2,3,4$.
Then 8 points on $X_{\Delta_2}(34)$ lying over $z_i$'s are represented by $z_i, [3]z_i$ for $i=1,2,3,4$.
$\widehat{W}_{2}$ fixes $z_i$ on $X_{\Delta_2}(34)$ if and only if
$W_{2,1}W_{2,i}^{-1}=[a]$ for some $a\in\Delta_2$.
By a direct computation, we have $a = 1,-1,15,9$ for $i=1,2,3,4$, respectively.
Thus $\widehat{W}_{2}$ fixes all of them.
$\widehat{W}_{2}$ fixes $[3]z_i$ on $X_{\Delta}(34)$ if and only if
$W_{2,1}[3]W_{2,i}^{-1}[3]^{-1}=[a]$ for some $a\in\Delta_2$.
By a direct computation, we have $a = 1,-1,15,9$ for $i=1,2,3,4$, respectively.
Therefore $\widehat{W}_{2}$ fixes all the 8 points lying over the 4 fixed points on $X_0(34)$ by $W_{2,i}$.
Thus $\widehat{W}_{2}$ is a bielliptic involution on $X_{\Delta_2}(34)$.
\end{Exm}

\begin{Rk}\label{45and64}
A completely analogous calculation shows that $\widehat{W}_9$ is a bielliptic involution
of $X_{\Delta_4}(45)$.
\par 
But for $X_{\Delta_3}(64)$, which is also of genus $5$ and $X_0(64)$ a non-hyperelliptic 
curve of genus $3$, the same setting produces a different outcome. Here $W_{64}$ is a bielliptic 
involution of $X_0(64)$, but calculating the elliptic elements we see that $\widehat{W}_{64}$ and
$[3]\widehat{W}_{64}$ both have $4$ fixed points each. So they are not bielliptic involutions of
$X_{\Delta_3}(64)$. (See Section \ref{sec:bielliptic} for the treatment of this case.)
\end{Rk}

In Example \ref{calcfixedpoints} we were lucky that the fixed points of $W_2$ on $X_0(34)$
couldn't be ramified in the covering $X_{\Delta_2}(34) \to X_0(34)$. Otherwise it might
be tedious to determine which representatives lying over the same $[z_i]$ give the same 
point on $X_{\Delta}(N)$. Luckily, this happens very rarely.

\begin{Lem}
\begin{itemize} 
\item[(a)] If $d\neq 3$, the non-cuspidal fixed points of $W_d$ on $X_0(N)$ are unramified
in the covering $X_{\Delta}(N) \to X_0(N)$.
\item[(b)] If $|(\Z/N\Z)^*/\Delta|$ is not divisible by $3$, the fixed points of $W_3$ on
$X_0(N)$ are unramified in the covering $X_{\Delta}(N) \to X_0(N)$.
\end{itemize} 
\end{Lem}

\begin{proof}
By \cite[p. 454]{O2} the non-cuspidal fixed points of $W_d$ on $X_0(N)$ correspond to certain 
elliptic curves with complex multiplication by orders with discriminants $-4d$ and maybe $-d$. 	
On the other hand, the only points that can possibly ramify in the covering $X_1(N)\to X_0(N)$
are cusps and elliptic points. The latter ones have complex multiplication by $\Z[i]$ (then
the ramification index is $2$ or $1$) or by $\Z[\zeta_3]$ (with ramification index $3$ or $1$). 
So a necessary condition for a fixed point of $W_d$ to be ramified in $X_{\Delta}(N)$ is $d=4$
or $d=3$, and for $d=3$ in addition the degree of the covering has to be divisible by $3$. 
This proves (b) and (a) for $d\neq 4$.
\par 
Now let $P$ be a fixed point of $W_4$ on $X_0(N)$. Choose the matrix $W_4$ such that it actually 
fixes a point $z\in{\mathbb H}$ representing $P$. Then $W_4$ must be an elliptic element and
consequently have trace $0$. So $W_4$ defines an involution of $X_{\Delta}(N)$. Note that
$W_4$ always lifts to $X_{\Delta}(N)$ by Lemma \ref{thm:whenaut}. If $P$ were ramified in
$X_{\Delta}(N)\to X_0(N)$ it would have ramification index $2$; so it would be fixed by an
involution $[a]$. But it is already fixed by the involution $W_4$ of $X_{\Delta}(N)$, and 
since the stabilizer of a point is always cyclic, it cannot contain $2$ involutions.
\end{proof}

\begin{Exm}
For $d=3$ there are indeed counter-examples. Consider the curve $X_{\Delta_1}(21)$ where
$\Delta_1 =\{\pm 1,\pm 8\}$. Then $W={6\ \ \ -1\choose 21\ \ -3}$ is an automorphism of
$X_{\Delta_1}(21)$ inducing $W_3$ on $X_0(21)$. Moreover, $W^2 =[5]$, so $W$ has order $6$.
One easily checks that $W$ fixes a point $z\in {\mathbb H}$. This gives rise to a point on 
$X_{\Delta_1}(21)$ that is fixed by $[5]$, 
so ramified in the covering $X_{\Delta_1}(21)\to X_0(21)$, 
and whose image in $X_0(21)$ is fixed by $W_3$.	 
\end{Exm}

\section{Non-bielliptic curves}\label{sec:non}

In this section, we exclude all the non-bielliptic curves $X_\Delta(N)$. 
For the extremal cases $\Delta=(\Z/N\Z)^*$ and $\Delta=\{\pm 1\}$, i.e.
$X_0(N)$ and $X_1(N)$, it is known which are bielliptic (\cite{B1}, \cite{J-K1});
so we can concentrate on the intermediate case. In other words, $\Delta$ is 
always meant to be strictly between $\{\pm 1\}$ and $(\Z/N\Z)^*$.
\par
Here, as in the rest of the paper, we use the notation 
$Aut(X)=Aut_{\overline{K}}(X)$.
\par 
When dealing with an individual curve, the following facts are very useful.

\begin{Thm}{\rm{\textbf{(Castelnuovo's Inequality)}}}
Let $F$ be a function field with constant field $k$ of characteristic zero.
Suppose there are two subfields $F_1$ and $F_2$ with constant
field $k$ satisfying
\begin{enumerate}
\item[(1)] $F=F_1F_2$ is the compositum of $F_1$ and $F_2$.
\item[(2)] $[F:F_i]=n_i$, and $F_i$ has genus $g_i$ $(i=1,2)$.
\end{enumerate}
Then the genus $g$ of $F$ is bounded by
$$g\leq n_1g_1+n_2g_2+(n_1-1)(n_2-1).$$
\end{Thm}

A proof can be found in \cite[Theorem III.10.3]{Sti}. Or see \cite[Theorem 3.5]{Ac2} for a proof in the language of Riemann surfaces.



\begin{Prop}\cite{Sch}\label{prop:unique}
Let $K$ be a field of characteristic $0$ and $X$ a curve defined over $K$ of genus 
$g\geq 6$. If $X$ is bielliptic, then the bielliptic involution $v$ is unique, defined 
over $K$ and lies in the center of $Aut(X)$.
\end{Prop}

\begin{proof}
Suppose $X$ has two bielliptic involutions, then the genus $g\leq 5$ due to Castelnuovo's Inequality.
Thus $g\geq6$ implies uniqueness. Conjugating $v$ by an element from $Aut(X)$ or letting 
$Gal(\overline{K}/K)$ act on $v$, we get another bielliptic involution; so the uniqueness
implies that $v$ is central and defined over $K$. 
\end{proof}

\begin{Rk}
Note however that even if the bielliptic involution $v$ of $X$ is defined over $K$ this does not 
always imply that the genus $1$ curve $X/\langle v\rangle$ is elliptic over $K$, as it might not 
have any $K$-rational points. This problem does of course not occur if $X$ has a $K$-rational point.
\end{Rk} 
 
Under certain condition we can now give a more precise version of Proposition \ref{image}.
 
\begin{Lem}\label{galoiscover} 
Let $X$ be a curve of genus $g(X)\ge 6$ and $\phi :X\to Y$ a finite Galois cover with group
$G=Gal(X/Y)$ such that $g(Y)\ge 2$. If $X$ has a bielliptic involution $v$, then:
\begin{itemize} 
	\item[(a)] $v\not\in G$.
	\item[(b)] $v$ commutes with all elements in $G$.
	\item[(c)] $G$ induces an isomorphic group $\widetilde{G}$ of automorphisms 
	of the genus $1$ curve $X/\langle v\rangle$.
	\item[(d)] $v$ induces a hyperelliptic or bielliptic involution $\widetilde{v}$ on $Y$.
\end{itemize} 	 
\end{Lem} 
 
\begin{proof}
(a) is clear because $g(Y)\ge 2$.  \\
(b) $v$ is central because $g(X)\ge 6$.\\
(c) follows from (b).\\
(d) From (b) it is clear that $v$ induces an involution $\widetilde{v}$ on $Y$.
As $Y/\langle\widetilde{v}\rangle =X/\langle G,v\rangle$ is a quotient of the
genus $1$ curve $X/\langle v\rangle$, we have $g(Y/\langle\widetilde{v}\rangle)\le 1$. 
\end{proof}

Note that among the list of the in total $76$ values of $N$ in 
Table \ref{classification} there are $21$ values for which 
there exist no intermediate curves $X_\Delta(N)$ namely for
$$N=1,2,\dots,12,14,18,22,23,46,47,59,83,94.$$
(See Table \ref{lowgenus}.) 
Under the condition $g_\Delta(N)\ge 2$, applying Proposition \ref{image} to Table \ref{classification}, 
we therefore have 
\begin{Lem}\label{lem:discard}
There are at most $48$ possible values of $N$ for which there might
exist an intermediate modular curve $X_\Delta(N)$ that is bielliptic, 
namely:
\begin{quote}
$21,$ $24-26,$ $28-45$, $48-51,$ $53-56,$ $60-65,$ 
$69,$ $71,$ $72,$ $75,$ $79,$ $81,$ $89,$ $92,$ $95,$ $101,$ $119,$ $131$
\end{quote}
\end{Lem}

For the reader's convenience, we tabulate in the Appendix in Tables \ref{lowgenus} 
and \ref{highgenus} all intermediate groups $\Delta_i$ for these values of $N$ 
together with their genera $g_{\Delta_i}(N)$ by using the genus formula in 
\cite[Theorem 1.1]{J-K3}.

First of all, we make a list of all $X_\Delta(N)$ that are rational, elliptic
or hyperelliptic: the rational and elliptic ones are listed in the Appendix in Table \ref{lowgenus} and the hyperelliptic ones are discussed in the next theorem. 

\begin{Thm}\label{thm:subhyper} 
$X_{\Delta_1}(21)$ is the only hyperelliptic curve.
\end{Thm}

\begin{proof}
	This is essentially shown in \cite{I-M} and \cite{J-K3}. However, the
	hyperellipticity of $X_{\Delta_1}(21)$ had been overlooked in \cite{I-M}.
	\par
	Moreover, the proof that $X_{\Delta_3}(37)$ is not hyperelliptic in
	\cite{I-M} is based on the claim from \cite{Mo} that this curve has
	no exceptional automorphisms. But we shall see in Lemma \ref{lem:37biell}
	and Theorem \ref{thm:Aut37} that $X_{\Delta_3}(37)$ actually has exceptional
	automorphisms. 
	\par
	So, concerning non-hyperellipticity, we want to provide a proof 
	at least for $N=37$. Obviously it suffices to treat the two maximal 
	cases $\Delta_3$ and $\Delta_4$. 
	\par
	The curve $X_{\Delta_3}(37)$ is of genus $4$ and it is a degree $3$ Galois
	cover of the genus $2$ curve $X_0(37)$. By the Hurwitz formula this implies 
	that the covering map is unramified. If $X_{\Delta_3}(37)$ were hyperelliptic,
	$Gal(X_{\Delta_3}(37)/X_0(37))$, a group of order $3$, would act without 
	fixed points on the $10$ fixed points of the hyperelliptic involution.
	\par
	The automorphism group of $X_{\Delta_4}(37)$ has the subgroup 
	$\langle [2], \widehat{W}_{37}\rangle$, of order $4$, whose quotient 
	is the elliptic curve $X_0^+(37)$. If $X_{\Delta_4}(37)$ were hyperelliptic, 
	this group would act on the $10$ fixed points of the hyperelliptic 
	involution. So at least one of these points would be fixed by the 
	hyperelliptic involution and another involution, contradicting the 
	fact that the stabilizer of a point is always cyclic.
\end{proof}


Now we start to apply various criteria that can show non-biellipticity.

\subsection{Cusps}
We apply Lemma \ref{galoiscover} to $X=X_{\Delta}(N)$ and $Y=X_0(N)$.

\begin{Lem}\label{lem:cusp1}
The curves $X_\Delta(N)$ are not bielliptic curves for the following $N:$
\begin{quote}
$31,$ $43,$ $53,$ $61,$ $65,$ $71,$
$75,$ $79,$ $89,$ $95,$ $101,$ $119,$ $131$
\end{quote}
Note that for these $N$ we have $g_\Delta(N)\geq 6$ for all possible $\Delta$.
\end{Lem} 

\begin{proof}
Let $N$ be one of the numbers of the above list.
Then  $g(X_0(N))\ge 2$ and $X_0(N)$ is either hyperelliptic or bielliptic, but not both.
From the tables in \cite{B1,O2}, we know that the hyperelliptic or bielliptic involution 
is the Fricke involution $W_N$. Suppose that $X_\Delta(N)$ is bielliptic.
Then by Lemma \ref{galoiscover} bielliptic involution $v$ induces $\widetilde{v}=W_N$ on 
$X_0(N)$; so $v$ must be of the form $[a]\widehat{W}_N$. This is a contradiction, because
$v$ must be defined over $\Q$ by Proposition \ref{prop:unique}, but $[a]\widehat{W}_N$ is 
not by Lemma \ref{Frickeinvolutions}. 
\end{proof}

Ultimately this proof is based on the fact that $[a]\widehat{W}_N$ maps some rational 
cusps to non-rational cusps. (Compare the proof of Lemma \ref{Frickeinvolutions}.)
Using part (a) of Lemma \ref{lem:cuspsfieldofdef} instead of
part (c), this can also be applied to some other cases with partial Atkin-Lehner involutions,
but whether the criterion works will then also depend on $\Delta$. 

\begin{Lem}\label{lem:cusp2} The following $X_\Delta(N)$ with $g_\Delta(N)\ge 6$ are not bielliptic:
\begin{quote}
$X_{\Delta_1}(29)$, $X_{\Delta_1}(33)$, $X_{\Delta_1}(35)$, $X_{\Delta_2}(35)$,
$X_{\Delta_1}(38)$, $X_{\Delta_2}(39)$, $X_{\Delta_1}(41)$, $X_{\Delta_2}(41)$,
$X_{\Delta_3}(41)$, $X_{\Delta_1}(42)$, $X_{\Delta_2}(42)$, $X_{\Delta_1}(51)$,
$X_{\Delta_2}(51)$, $X_{\Delta_3}(51)$, $X_{\Delta_1}(55)$, $X_{\Delta_2}(55)$,
$X_{\Delta_3}(55)$, $X_{\Delta_1}(60)$, $X_{\Delta_2}(60)$, $X_{\Delta_3}(60)$, 
$X_{\Delta_4}(60)$, $X_{\Delta_5}(60)$, $X_{\Delta_6}(60)$, $X_{\Delta_1}(62)$, 
$X_{\Delta_2}(62)$, $X_{\Delta_1}(69)$, $X_{\Delta_1}(92)$, $X_{\Delta_2}(92)$.
\end{quote}
\end{Lem}

\begin{proof}
Suppose that a curve $X_\Delta(N)$ from the above list is bielliptic.
Since $g_\Delta(N)\geq 6$, $X_\Delta(N)$ has a unique bielliptic involution 
$v$ which induces an involution $\tilde v$ on $X_0(N)$. Note that by \cite{O2}
and \cite{B1} any hyperelliptic or bielliptic involution on the corresponding 
$X_0(N)$ is equal to one of the Atkin-Lehner involutions $W_d$ with $d\neq 2$.
Thus $\tilde v$ should be $W_d$ with $d\neq 2$.
Note that $W_d$ is represented by a matrix
$\begin{pmatrix} dx & y\\ Nz & dw\end{pmatrix}$
where $x,y,z,w\in\mathbb Z$ and $\det W_d=d.$ Because of $d|N$ this implies 
$(y,d)=1$. Furthermore, we can choose $w=1$. Then $\tilde v$ maps the cusp
$0$ to $\sm y\\d\esm.$
By using Lemma \ref{lem:cuspsfieldofdef}, one can check that for the given 
values of $N$, $\Delta_i$ and $d$ the cusps lying above $\sm y\\d\esm$ are 
non-rational. Thus $v$ maps rational cusps to non-rational cusps. This gives 
rise to a contradiction.
\par 
For example $X_0(39)$ has a hyperelliptic involution $W_{39}$ and a unique bielliptic 
involution $W_3$. If $\widetilde{v}=W_3$, then $v$ maps cusps above $0$ to cusps above
${1\choose 3}$. For $\Delta_2 =\{\pm 1, \pm 16, \pm 17\}$ the latter ones are only defined 
over $\Q(\zeta_3)$ because $\Delta_2^{(3)}=\{1\}$. 
\end{proof}

Note that in contrast to $X_{\Delta_2}(39)$ for $X_{\Delta_3}(39)$ this approach does not work,
because then the cusps of $X_{\Delta_3}(39)$ above ${y\choose 3}$ are rational. 

\begin{Lem}\label{lem:39and40} The modular curves $X_{\Delta_3}(39)$, $X_{\Delta_4}(40)$ and 
$X_{\Delta_5}(40)$ are not bielliptic.
\end{Lem}

\begin{proof} Suppose $X_{\Delta_3}(39)$ is bielliptic. Since 
$X_{\Delta_3}(39)$ is of genus $9$, the bielliptic involution 
$v$ is unique and must have $16$ fixed points. The induced 
involution $\tilde v$ on $X_0(39)$ must be equal to $W_3$ or 
$W_{39}$ \cite{B1,O2}. By the same reason as in the proof of 
Lemma \ref{lem:cusp1}, $\tilde v$ is not $W_{39}$. Suppose 
$\tilde v$ is equal to $W_3$. Note that $W_3$ is the only 
bielliptic involution on $X_0(39)$ and it has 4 fixed points 
\cite{B1}. Since the covering $X_{\Delta_3}(39)\to X_0(39)$ is 
of degree $3$, $v$ has at most $12$ fixed points on 
$X_{\Delta_3}(39)$, which is a contradiction.
\par 
For the same reason a bielliptic involution $v$ of $X_{\Delta_4}(40)$ or 
$X_{\Delta_5}(40)$ (both of genus $7$ and mapping with degree $2$ to 
$X_0(40)$, which has genus $3$) cannot induce a bielliptic involution
$\widetilde{v}$ on $X_0(40)$. 
\par 
The hyperelliptic involution of $X_0(40)$ is 
$w=\begin{pmatrix} -10&1\\-120&10\end{pmatrix}$ 
which is not of Atkin-Lehner type (cf. \cite{O2}).
Note that $w\sm 1\\4\esm =\infty$. 
By using \cite[Lemma 1.2]{I-M}, we know that the cusps of 
$X_{\Delta_4}(40)$ (resp. $X_{\Delta_5}(40)$) lying above $\sm 1\\4\esm$
are all rational. This can also easily be seen directly: The only possible
Galois conjugate of ${1\choose 4}$, namely ${-1\choose 4}$, gives the same 
cusp, because 
${-1\choose 4}={-9\ \ \ 2\choose 40\ \ -9}{1\choose 4}=[-9]{1\choose 4}$ 
with
$[-9]\in\Gamma_{\Delta_i}(40)$ ($i=4,5$).
So if $v$ induces $w$, we again get a contradiction from $v$ mapping rational 
cusps to non-rational ones.
\end{proof}

\subsection{Unramifed coverings}

The argumentation with the fixed points of the bielliptic involution in the proof 
of Lemma \ref{lem:39and40} can be used in a more systematic way. 

\begin{Lem}\label{lem:fp-method} 
Let $X$ be a curve of genus $g(X)\ge 6$ and $\phi:X\to Y$ a finite Galois covering
with $G=Gal(X/Y)$ such that $Y$ is not subhyperelliptic.
If $X$ has a bielliptic involution $v$, then the following must all hold:
\begin{itemize}
\item[(a)] $v$ induces a bielliptic involution $\widetilde{v}$ on $Y$.
\item[(b)] The only points of $X$ that are ramified in the covering
$X\to X/\langle G,v\rangle$ are the fixed points of $v$ and they have 
ramification index $2$.
\item[(c)] Every automorphism in $\langle G,v\rangle$ other than $v$ and the identity
is fixed point free.
\item[(d)] $\phi$ is totally unramified.
\item[(e)] $g(X)-1=\deg(\phi)(g(Y)-1)$.
\end{itemize} 
\end{Lem} 

\begin{proof}
(a) is clear from Lemma \ref{galoiscover}, because $Y$ is not subhyperelliptic. Now 
$\widetilde{\phi}:X/\langle v\rangle \to X/\langle v,G\rangle=Y/\langle	\widetilde{v}\rangle$
as a covering of genus $1$ curves must be totally unramified. This immediately implies (b).
If $\sigma\in\langle G,v\rangle$ has a fixed point $P$ on $X$, then $P$ must be a fixed point
of $v$ by (b). Moreover, $\sigma\not\in\langle v\rangle$ would imply that the ramification 
index of $P$ is bigger than $2$. This proves (c). 
\par 
As a special case of (c), every element in $G$ is fixed point free, which means that 
$X\to X/G$ is totally unramified, in other words, (d) holds.
Finally, (e) is equivalent to (d) by the Hurwitz formula. 
\end{proof} 
 
Lemma \ref{lem:fp-method} is a useful criterion, as it is often quite easy to see that condition (e) is not satisfied.

\begin{Lem}\label{lem:fp-app} 
The following $X_\Delta(N)$ are not bielliptic:
\begin{quote}
$X_{\Delta_1}(37)$, $X_{\Delta_2}(37)$, $X_{\Delta_2}(40)$, $X_{\Delta_3}(40)$, $X_{\Delta_1}(44)$,
$X_{\Delta_2}(44)$, $X_{\Delta_3}(45)$, $X_{\Delta_4}(48)$,
$X_{\Delta_5}(48)$, $X_{\Delta_6}(56)$,
$X_{\Delta_7}(56)$, $X_{\Delta_6}(63)$, $X_{\Delta_7}(63)$,
$X_{\Delta_8}(63)$, $X_{\Delta_2}(64)$, $X_{\Delta_6}(72)$,
$X_{\Delta_7}(72)$.
\end{quote}
\end{Lem}

For each $X=X_\Delta(N)$ in the list of Lemma \ref{lem:fp-app}, we
suggest a finite Galois covering $\phi:X\to Y$ and its degree in 
Table \ref{tb:fp-app} which enables us to conclude from criterion 
(e) of Lemma \ref{lem:fp-method} that $X$ is not bielliptic.

\begin{center}
\begin{longtable}{c|c}
\caption{List of maps $\phi:X\to Y$ and their degrees}
\label{tb:fp-app}\\ $\phi:X\to Y$ & degree
 \\ \hline
 $X_{\Delta_1}(37)\to X_{\Delta_3}(37)$ & 3
 \\ \hline
 $X_{\Delta_2}(37)\to X_{\Delta_3}(37)$ & 2
 \\ \hline
 $X_{\Delta_2}(40)\to X_{\Delta_6}(40)$ & 2
 \\ \hline
 $X_{\Delta_3}(40)\to X_1(20)$ & 2
 \\ \hline
 $X_{\Delta_1}(44)\to X_1(22)$ & 2
 \\ \hline
 $X_{\Delta_2}(44)\to X_0(44)$ & 2
 \\ \hline
 $X_{\Delta_3}(45)\to X_0(45)$ & 3
 \\ \hline
 $X_{\Delta_4}(48)\to X_{\Delta_1}(24)$ & 2
 \\ \hline
 $X_{\Delta_5}(48)\to X_{\Delta_2}(24)$ & 2
 \\ \hline
 $X_{\Delta_6}(56)\to X_0(56)$ & $2$
 \\ \hline
 $X_{\Delta_7}(56)\to X_0(56)$ & $2$
 \\ \hline
 $X_{\Delta_6}(63)\to X_0(63)$ & $3$
 \\ \hline
 $X_{\Delta_7}(63)\to X_0(63)$ & $3$
 \\ \hline
 $X_{\Delta_8}(63)\to X_0(63)$ & $3$
 \\ \hline
 $X_{\Delta_2}(64)\to X_{\Delta_3}(64)$ & $2$
 \\ \hline
 $X_{\Delta_6}(72)\to X_0(72)$ & $2$
 \\ \hline
 $X_{\Delta_7}(72)\to X_0(72)$ & $2$
 \\ \hline
\end{longtable}
\end{center}


\subsection{Castelnuovo's inequality}

Consider $X_{\Delta_1}(34)$ of genus 9. Note that there is a
natural map $\phi: X_{\Delta_1}(34)\to X_{\Delta_1}(17)$ of degree
$3$, and $X_{\Delta_1}(17)$ is of genus 1. Suppose
$X_{\Delta_1}(34)$ is bielliptic. Then there is a map of degree
$2$ from $X_{\Delta_1}(34)$ to an elliptic curve $E.$ Let
$F$ (resp. $F_1,F_2$) be the function field of
$X_{\Delta_1}(34)$ (resp. $X_{\Delta_1}(17),E$). Applying
Castelnuovo's inequality, we get a contradiction. Thus
$X_{\Delta_1}(34)$ is not bielliptic. By using the same argument,
we have the following result:

\begin{Lem}\label{lem:castel-method} 
The following $X_\Delta(N)$ are not bielliptic:
\begin{quote}
$X_{\Delta_1}(34)$, $X_{\Delta_2}(45)$,
$X_{\Delta_1}(49)$, $X_{\Delta_1}(50)$, $X_{\Delta_1}(54)$,
$X_{\Delta_5}(56)$, $X_{\Delta_{10}}(63)$, $X_{\Delta_5}(72)$,
$X_{\Delta_8}(72)$, $X_{\Delta_2}(81)$.
\end{quote}
\end{Lem}

For each $X=X_\Delta(N)$ in the list of Lemma \ref{lem:castel-method}, 
we suggest a finite morphism $\phi:X\to Y$ and its degree in Table 
\ref{tb:castel} which enables us to conclude that $X$ is not bielliptic 
by applying Castelnuovo's inequality.

\begin{center}

\begin{longtable}{c|c}
\caption{List of maps $\phi:X\to Y$ and their degrees}\label{tb:castel}\\
$\phi:X\to Y$ & degree
 \\ \hline
 $X_{\Delta_1}(34)\to X_{\Delta_1}(17)$ & 3
 \\ \hline
 $X_{\Delta_2}(45)\to X_1(15)$ & 3
 \\ \hline
 $X_{\Delta_1}(49)\to X_0(49)$ & 7
 \\ \hline
 $X_{\Delta_1}(50)\to X_{\Delta_1}(25)$ & 3
 \\ \hline
 $X_{\Delta_1}(54)\to X_{\Delta_1}(27)$ & 3
 \\ \hline
 $X_{\Delta_5}(56)\to X_{\Delta_1}(28)$ & 2
 \\ \hline
 $X_{\Delta_{10}}(63)\to X_{\Delta_2}(21)$ & 3
 \\ \hline
 $X_{\Delta_5}(72)\to X_0(24)$ & 9
 \\ \hline
 $X_{\Delta_8}(72)\to X_{\Delta_3}(24)$ & 3
 \\ \hline
 $X_{\Delta_2}(81)\to X_{\Delta_1}(27)$ & 3
 \\ \hline
\end{longtable}
\end{center}

\subsection{Elliptic elements}
We have a second, more careful look at Lemma \ref{lem:fp-method}.

\begin{Lem}\label{lem:elliptic} 
The following $X_\Delta(N)$ are not bielliptic:
\begin{quote}
$X_{\Delta_8}(56)$, $X_{\Delta_9}(63)$, $X_{\Delta_2}(69)$.
\end{quote}
\end{Lem}

\begin{proof} 	
We first treat the case $X_{\Delta_8}(56)$ in detail.
Suppose $v$ is a bielliptic involution on $X_{\Delta_8}(56)$ and
$\widetilde{v}$ is the induced involution on $X_0(56)$.
Note that $X_0(56)$ is not hyperelliptic but bielliptic, and all the
bielliptic involutions of $X_0(56)$ are $W_{56}$, $W_7$ and
$W_7S_2W_8S_2$ where $S_2=\begin{pmatrix}1&1/2\\0&1\end{pmatrix}$
(see \cite[Theorem 3.15]{B1}).
By the same argument as in the proof of Lemma \ref{lem:cusp1}, 
$\widetilde{v}$ cannot be $W_{56}$. 
\par 
With $\phi:X_{\Delta_8}(56)\to X_0(56)$ we are in the situation of Lemma \ref{lem:fp-method},
but condition (e) is satisfied. 	
\par 
For each of the other two bielliptic involutions $\widetilde{v}$ we give an elliptic element	
$w$ that normalizes $\Gamma_{\Delta_8}(56)$ and gives $\widetilde{v}$ on $X_0(56)$, and we give
a matrix $[a]\in (\Z/56\Z)^*\setminus\Delta_8$ such that $[a]w$ is also an elliptic element.
The logic behind this construction is the following:
If $v$ induces $\widetilde{v}$ on $X_0(56)$, then $v$ and $w$ are automorphisms of 
$X_{\Delta_8}(56)$ that agree on $X_0(56)$; so they differ by an element $[b]$ from
$G=Gal(X_{\Delta_8}(56)/X_0(56))$, i.e. $v=[b]w$. Now $w$ and $[a]w$ both are in 
$\langle G,v\rangle$, and at least one of them is different from $v$. But, being represented
by elliptic elements, both of them have fixed points. So condition (c) of Lemma \ref{lem:fp-method}
cannot hold.
\par 
One easily checks that $S_2$ normalizes $\Gamma_{\Delta_8}(56)$, 
as do $W_7$ and $W_8$. So $W_7S_2W_8S_2$ also normalizes
$\Gamma_{\Delta_8}(56)$.  
\par 
Take
$W_{7}=\begin{pmatrix} 7&-1\\56&-7\end{pmatrix}$ and
$[5]=\begin{pmatrix} 5&-1\\56&-11\end{pmatrix}$. Then
$[5]W_7=\begin{pmatrix} -21&2\\224&21\end{pmatrix}$. By condition (c) of Lemma
\ref{lem:fp-method}, $\tilde v$ cannot be $W_7$.  
Finally take
$W_{7}=\begin{pmatrix} 7&-6\\-56&49\end{pmatrix}$,
$W_{8}=\begin{pmatrix} -64&-3\\-168&-8\end{pmatrix}$ and
$[5]=\begin{pmatrix} 61&6\\-112&-11\end{pmatrix}$. Then
$W_7S_2W_8S_2=\begin{pmatrix} -28&-15\\56&28\end{pmatrix}$ and
$[5]W_7S_2W_8S_2=\begin{pmatrix}
-1372&-747\\2520&1372\end{pmatrix}$.  
Thus $\tilde v$ also cannot be
$W_7S_2W_8S_2$, and hence $X_{\Delta_8}(56)$ is not bielliptic.
\par
The other curves are dealt with in the same manner. Note that $X_0(63)$ has exceptional 
automorphisms. But by \cite[Theorem 3.15]{B1} all its bielliptic involutions come from
$PSL_2(\R)$, namely $W_{63}$, $W_7 S_3^2 W_9 S_3$ and $W_7 S_3^2 W_9 S_3$ where
$S_3=\begin{pmatrix}1&1/3\\0&1\end{pmatrix}$.
\par 
As usual $W_{63}$ is not possible.
If we take $[2]=\begin{pmatrix} 65 & -8\\ 252 & -31 \end{pmatrix}$, $W_7=\begin{pmatrix}-7 & 4 \\ -63 & 35\end{pmatrix}$, $W_9=\begin{pmatrix} -9 & 5 \\ 63 & -36 \end{pmatrix}$,
then $W_7 S_3^2 W_9 S_3= \begin{pmatrix} 21 & -4 \\ 126 & -21 \end{pmatrix}$ and $[2] W_7 S_3^2 W_9 S_3=\begin{pmatrix} 357 &-92 \\ 1386 & -357 \end{pmatrix}$.
Also if we take $[-31]=\begin{pmatrix} -31& -8 \\ 252 & 65 \end{pmatrix}$, $W_7=\begin{pmatrix} -14 & 3 \\ 63 &-14 \end{pmatrix}$, $W_9=\begin{pmatrix} 9& -5 \\ -63& 36 \end{pmatrix}$,
then $W_7 S_3^2 W_9 S_3=\begin{pmatrix} -21& -4 \\ 126 & 21 \end{pmatrix}$ and $[-31] W_7 S_3 W_9 S_3^2=\begin{pmatrix} -357& -44 \\ 2898& 357\end{pmatrix}$.
\par 
Finally for $X_{\Delta_2}(69)$ we take 
$[22]=\begin{pmatrix} -47 & 15 \\ -69 & 22\end{pmatrix}$ and $W_{23}=\begin{pmatrix} 23& -8 \\ 69& -23 \end{pmatrix}$;
then $[22] W_{23}=\begin{pmatrix} -46& 31 \\ -69 & 46\end{pmatrix}$.
%
\end{proof}

%
%
%
%

\subsection{Bring's curve}

Next we will treat the modular curve $X_{\Delta_1}(25)$. Kubert 
denoted this curve by $B$ in \cite{Ku}. We first thought that 
Kubert used such a notation because it would be isomorphic to
Bring's curve. Bring's curve is the unique curve of genus $4$ 
with full automorphism group $S_5$. Also it is isomorphic to 
the modular curve $X_1(5,10)$ and it is a bielliptic curve 
(cf. \cite{C-D,Hu}). But it turns out that $X_{\Delta_1}(25)$ 
is not Bring's curve and not even a bielliptic curve.

\begin{Thm}\label{bring} 
Bring's curve is the unique bielliptic curve of genus $4$ over $\C$
that has an automorphism of order $5$.
\end{Thm}

\begin{proof} 
The automorphism $\sigma$ of order $5$ acts by conjugation on the 
bielliptic involutions. If it fixes one of them, this induces an 
automorphism of order $5$ with fixed points on a curve of genus 
$1$, which is impossible. So the number of bielliptic involutions 
must be divisible by $5$. By \cite[Corollary 6.9]{C-D} this number 
cannot be $5$, and $10$ means Bring's curve.
\end{proof}

\begin{Lem}\label{lem:delta25}
$X_{\Delta_1}(25)$ is not a bielliptic curve.
\end{Lem}

\begin{proof} 
Obviously $[6]$ is an automorphism of $X_{\Delta_1}(25)$ of order $5$. So 
if the curve were bielliptic, it would be isomorphic to Bring's curve 
$X_1(5,10)$.
Both curves are defined over some number field, so an isomorphism 
would be defined over the algebraic closure of $\mathbb Q$ and 
hence over a suitable number field.
By conjugating with the diagonal matrix 
$\begin{pmatrix}5&0\\0&1\end{pmatrix}$,
one gets an isomorphism between $X_1(5,10)$ and $X_{\Delta_2}(50)$.
The latter curve covers $X_0(50)$.
By Cremona's table \cite{Cr} there are elliptic curves over $\mathbb Q$ with 
conductor $50$. These have multiplicative reduction at $2$ and hence also 
multiplicative reduction over any number field at any prime above $2$.
They are isogeny factors of the Jacobian of the modular curve isomorphic 
to $X_1(5,10)$, which therefore must have bad reduction above $2$. On the 
other hand, $X_{\Delta_1}(25)$ is covered by $X_1(25)$ and therefore has 
good reduction outside the primes above $5$.
\end{proof}

\subsection{The case $X_{\Delta_4}(37)$}


\begin{Lem}\label{lem:delta37} 
The group $\langle [2], \widehat{W}_{37}\rangle\cong C_2 \times C_2$ 
is a Sylow $2$-subgroup of $Aut(X_{\Delta_4}(37))$. In particular, 
$X_{\Delta_4}(37)$ is not bielliptic.
\end{Lem}

\begin{proof}
To ease notation we write $X$ for $X_{\Delta_4}(37)$ and $A$ for its full 
automorphism group. Also, $\langle [2], \widehat{W}_{37}\rangle$ will 
be denoted by $M$. 
\par 
We first show that $X_{\Delta_4}(37)$ is bielliptic if and only if 
$|A|$ is divisible by $8$.
\par
Each of the three involutions in $M$ has exactly $2$ fixed points. 
So if $X$ has a bielliptic involution $v$, the $2$-Sylow subgroup 
containg $M$ must also contain a conjugate of $v$, and hence its 
order must be divisible by $8$.
\par
Conversely, if $8$ divides $|A|$, then, by the theory of $p$-groups, 
there exists a group $D\subseteq A$ of order $8$ that contains $M$. 
Since $D$ has at least $3$ involutions, it cannot be cyclic or 
a quaternion group. On the other hand, the groups $C_4 \times C_2$ 
and $C_2 \times C_2 \times C_2$ cannot act as automorphisms on a genus 
$4$ Riemann surface \cite[Proposition 2]{Ki}. So $D$ must be a dihedral 
group $D_4$. Applying the Hurwitz formula to the covering $X\to X/D$ 
shows that the two involutions $v_1$, $v_2$ outside $M$ are bielliptic.
\par
For the same reason as in the proof of Lemma \ref{lem:delta25} the curve
$X$ is not Bring's curve. So by \cite[Corollary 6.9]{C-D} it has at most 
$6$ bielliptic involutions. Thus there exists a Galois extension $K$ of $\Q$ 
with $Gal(K/\Q)$ a (not necessarily transitive) subgroup of the symmetric
group $S_6$ such that all bielliptic involutions are defined over $K$.
On the other hand, $[2]$ is defined over $\Q$, whereas $\widehat{W}_{37}$ and 
$[2]\widehat{W}_{37}$ are defined over $\Q(\sqrt{37})$. 
\par
By using the computer algebra system SAGE, we can get the $q$-expansions 
of a basis $\{f_1,f_2,f_3,f_4\}$ of the cusps forms of weight $2$ for $\Gamma_{\Delta_4}(37)$.
Using these $q$-expansions and following the method described in 
Section 2 in \cite{Ha-Sh}, one can obtain a model of the curve $X_{\Delta_4}(37)$ over $\C$ as follows.
The curve $X_{\Delta_4}(37)$ can be identified with the canonical curve which is the image
of the canonical embedding
$$X_{\Delta_4}(37)\ni P\mapsto [f_1(P) :\cdots : f_4(P)] \in \mathbb P^3.$$
By Petri's Theorem, a minimal generating system of the ideal $I(X_{\Delta_4}(37))$ contains
a cubic polynomial and $X_{\Delta_4}(37)$ is not isomorphic to a smooth plane quintic curve.
(This follows from Petri's Theorem, see e.g. \cite[Theorem 2.1]{Ha-Sh}.)
To obtain a minimal generating system of  $I(X_{\Delta_4}(37))$; we only have to compute
the linear relations of the $f_if_j$ and $f_if_jf_k$ $(1\leq i,  j, k \leq 4)$ by using their $q$-expansions.
By this method, we obtain the following defining equations of $X_{\Delta_4}(37)$.
$$Q_1:\ \ -x_2 ^2 +x_3 x_1 -2x_4 x_3 ,$$

$$Q_2:\ \ 9x_2 ^2 x_1 -20x_2 ^3 -9x_3 x_1 ^2 +12x_3 x_2 x_1 
-8x_3 x_2 ^2 +16x_3 ^2 x_1 -8x_3 ^2 x_2 $$
$$ -12x_3 ^3 +8x_4 x_1 ^2 +18x_4 x_2 ^2 -20x_4 ^2 x_1 -24x_4 ^2 x_2 +24x_4 ^3 .$$


The curve $C$ defined by $Q_1$ and $Q_2$ is already defined over 
$\Q$ and smooth at $p=5$. Let $\overline{C}$ denote its reduction 
modulo $5$. 
\par
We emphasize that the question whether over $\Q$ the curve $C$ is 
already a model of $X_{\Delta_4}(37)$ or rather of one of its twists 
is not important for us. We only need the connection over $\C$, 
which implies $Aut_{\C}(X_{\Delta_4}(37))\cong Aut_{\C}(C)$. Moreover, 
$Aut_{\C}(C)$ embeds into $Aut_{\overline{\F_5}}(C)$. 
\par
So if $|A|$ is divisible by $8$, or equivalently $X$ is bielliptic, 
then by the preceding discussion $Aut_{\overline{\F_5}}(\overline{C})$ 
contains a pair of conjugate, non-commuting bielliptic involutions 
that generate a group $D_4$. 
By the preceding argumentation, these two bielliptic involutions are 
defined over the residue field of $K$, which is an extension of $\F_5$
whose Galois group is a cyclic subgroup of $S_6$. So they must both be 
defined over some field $\F_{5^k}$ with $k\leq 6$. 
\par
But using the computer algebra system MAGMA one can easily calculate 
for $k\leq 12$ that $|Aut_{\F_{5^k}}(\overline{C})|$ equals $2$ if $k$ 
is odd and $4$ if $k$ is even.

%

This finally proves the lemma.
\end{proof}

Having come so far, we cannot resist working out the full automorphism
group of $ X_{\Delta_4}(37)$.

\begin{Prop}\label{prop:37noex} 
The modular curve $X_{\Delta_4}(37)$ has no exceptional automorphisms.
In other words, $Aut(X_{\Delta_4}(37))=\langle [2], \widehat{W}_{37}\rangle$.
\end{Prop}

\begin{proof}
We keep the notation from the proof of the previous lemma, where
we have shown that $M=\langle [2], \widehat{W}_{37}\rangle$ is 
a $2$-Sylow subgroup of $A=Aut(X_{\Delta_4}(37))$. As a consequence 
the only possibilities for $|A|$ are $4$, $12$, $20$, $36$, and 
$60$. (See for example \cite{Ki} or \cite{K-K} for the possible 
orders of automorphism groups for genus $4$.)
\par
Also note that by \cite[Corollary 1, p. 346]{Ac1} an unramified 
degree $3$ Galois cover of a genus $2$ curve is bielliptic. So 
every element of order $3$ in $A$ must have $3$ or $6$ fixed points.
\par
If the normalizer of $M$ in $A$ is bigger than $M$, then there is
an element of order $3$ or $5$ in $A$ inducing an automorphism with
fixed points on the genus $1$ curve $X/M=X_0^+(37)$. But this is not 
possible. (Note that the $j$-invariant of $X_0^+(37)$ is not $0$.)
\par
So we can assume from now on that $M$ is its own normalizer. Then by 
Burnside's Transfer Theorem (see for example \cite[Theorem 14.3.1]{Ha})
$A$ contains a normal subgroup $N$ of index $4$. If for example $|A|=60$,
then $N$, being of order $15$, is cyclic, and hence $A$ has a unique
$5$-Sylow group $F$. An involution can only act as identity or as 
inversion on $F$. Applying the Galois automorphism of 
$\Q(\sqrt{37})/\Q$ to $A$ maps $F$ into itself and 
$\widehat{W}_{37}$ to $[2]\widehat{W}_{37}$ (compare the paragraph 
after Lemma \ref{Frickeinvolutions}). This shows that 
$\widehat{W}_{37}$ and $[2]\widehat{W}_{37}$ must act on $F$ in 
the same manner. Consequently $[2]$ commutes with the elements in 
$F$, and we end up with an automorphism of order $5$ on the curve
$X_0(37)$.
The same arguments apply if $N$ has order $3$ or $5$ or is a cyclic
group of order $9$. 
\par
So we are left with the case $N\cong C_3 \times C_3$. Then by the 
Hurwitz formula each nontrivial element in $N$ has $3$ fixed points. 
The action of $N$ cuts up the Jacobian $J$ of $X$ into a product
(up to isogeny) of $4$ elliptic curves $E_i$, each $E_i$ being 
the quotient of $X$ by an element $t_i$ of order $3$, and 
$N/\langle t_i\rangle$ inducing an automorphism of order $3$ with 
fixed points on $E_i$. This is only possible if $j(E_i)=0$. So $J$
must have potentially good reduction everywhere. This contradicts
$J$ having the isogeny factor $X_0^+(37)$ with multiplicative 
reduction at places above $37$.
\end{proof}

\subsection{The remaining cases}

By using Proposition \ref{image} together with all criteria above, we get the following result.

\begin{Cor}\label{cor:all} The following $X_\Delta(N)$ are not bielliptic:
\begin{quote}
$X_{\Delta_1}(39)$, $X_{\Delta_1}(40)$, $X_{\Delta_1}(45)$, $X_{\Delta_1}(48)$,
$X_{\Delta_2}(48)$, $X_{\Delta_3}(48)$, $X_{\Delta_1}(56)$, $X_{\Delta_2}(56)$,
$X_{\Delta_3}(56)$, $X_{\Delta_4}(56)$, $X_{\Delta_1}(63)$, $X_{\Delta_2}(63)$, 
$X_{\Delta_3}(63)$, $X_{\Delta_4}(63)$, $X_{\Delta_5}(63)$, $X_{\Delta_1}(64)$,
$X_{\Delta_1}(72)$, $X_{\Delta_2}(72)$, $X_{\Delta_3}(72)$, $X_{\Delta_4}(72)$,
$X_{\Delta_1}(81)$.
\end{quote}
\end{Cor}
\begin{proof}
We observe from Table \ref{highgenus} that
each curve $X_{\Delta_i}(N)$ listed in the statement is mapped to
$X_{\Delta_j}(N)$ which is neither subhyperelliptic nor  bielliptic for some $\Delta_j$ containing $\Delta_i$.
It then follows from Proposition \ref{image} that $X_{\Delta_i}(N)$ is not bielliptic, either. For example, 
$X_{\Delta_1}(39)$ is mapped to $X_{\Delta_3}(39)$, which is not bielliptic by Lemma \ref{lem:39and40}. 
Thus Proposition \ref{image} implies that $X_{\Delta_1}(39)$ is not bielliptic, either.
\end{proof}

As for the last case, $X_{\Delta_1}(36)$, the first two authors \cite[Lemma 2.6]{J-K2} already proved 
that $X_{\Delta_1}(36)$ is not bielliptic by using properties of the 
normalizer of $\Gamma_{\Delta_1}(36)$.

\begin{Lem}\cite[Lemma 2.6]{J-K2}\label{lem:delta36} 
	The modular curve $X_{\Delta_1}(36)$ is not bielliptic.
\end{Lem}

\begin{Rk}\label{whynot}
	The reader might wonder why we didn't systematically use \cite{J-K2},
	which decides for every curve $X_1(M,N)$ whether it is bielliptic or not.
	After all, conjugating $\Gamma_1(M,N)$ with ${M\ 0\choose 0\ 1}$, $X_1(M,N)$ 
	is isomorphic to some $X_{\Delta}(MN)$ with 
	$\Delta=\{a\in (\Z/NZ)^*\ :\ a\equiv \pm 1\ mod\ N\}$. 
	The answer is that this would mainly apply to curves which we have settled 
	in Corollary \ref{cor:all} as corollaries to other cases that are not 
	isomorphic to any $X_1(M,N)$, and that hence had to be dealt with anyway.
\end{Rk}

\section{Bielliptic curves}\label{sec:bielliptic}

In this section, we will show that the remaining curves $X_\Delta(N)$ 
are bielliptic. 
By \cite{J-K2} the curves $X_1(2,14)$, $X_1(2,16)$, $X_1(3,12)$, 
$X_1(4,12)$, $X_1(5,10)$, $X_1(7,7)$, and $X_1(8,8)$ are bielliptic. 
The first four are isomorphic (over $\C$) to $X_{\Delta_1}(28)$, 
$X_{\Delta_1}(32)$, $X_{\Delta_2}(36)$, $X_{\Delta_6}(48)$, respectively.
The other three are already known under different names.

\begin{Rk}\label{extremal}
The following curves are famous extremal examples.
\begin{itemize}
\item 
$X_1(5,10)$, which over $\C$ is isomorphic to $X_{\Delta_2}(50)$, is
Bring's curve, which we already discussed in the previous section.
It has $10$ bielliptic involutions, and is the unique curve of genus 
$4$ with more than $6$ bielliptic involutions. 
(See \cite[Corollary 6.9]{C-D}.)
\item 
$X(7)$, which over $\C$ is isomorphic to $X_{\Delta_2}(49)$, is the 
famous Klein quartic. Its automorphism group is the simple group 
$PSL_2(\F_7)$ of order $168$, and it has $21$ bielliptic involutions, 
the maximum possible for a curve of genus $3$.
\item 
$X(8)$, which over $\C$ is isomorphic to $X_{\Delta_3}(64)$, is the
Wiman curve, the unique curve of genus $5$ with the maximum possible
of $192$ automorphisms. It has exactly $3$ bielliptic involutions. 
(See \cite[Lemma 4.5]{B-K-X} and \cite{K-M-V}.)
\end{itemize}
See below for some methods to explicitly find some of their bielliptic
involutions.
\end{Rk}

Some of the remaining curves $X_{\Delta}(N)$ are easily seen to be 
bielliptic, as they are double covers of an elliptic curve $X_0(N)$,
to wit, $X_{\Delta_1}(24)$, $X_{\Delta_2}(24)$, and $X_{\Delta_2}(36)$.
Correspondingly, there is a bielliptic involution of the form $[a]$.
\par
Similarly, there is a degree $2$ map from $X_{\Delta_1}(32)$ to the
elliptic curve $X_{\Delta_2}(32)$ and a degree $2$ map from
$X_{\Delta_1}(28)$ to the elliptic curve $X_1(14)$.
\par
More generally, there can be a degree $2$ map from $X_{\Delta}(N)$
to an elliptic curve $X_{\Delta'}(N/2)$. This happens for 
$X_{\Delta_6}(40)$, $X_{\Delta_6}(48)$, and $X_{\Delta_3}(64)$, mapping 
to $X_{\Delta_1}(20)$, $X_{\Delta_3}(24)$, $X_{\Delta_2}(32)$, respectively.
\par
The curves $X_{\Delta_1}(30)$, $X_{\Delta_2}(33)$, $X_{\Delta_4}(35)$,
$X_{\Delta_4}(39)$, $X_{\Delta_6}(40)$, $X_{\Delta_4}(41)$, $X_{\Delta_6}(48)$,
are of genus $5$ and double covers of the hyperelliptic genus $3$ 
curve $X_0(N)$. By \cite[p. 50]{Ac2} they must be hyperelliptic 
(which is excluded by Theorem \ref{thm:subhyper}) or bielliptic.
Consequently, in these cases the two lifts of the 
hyperelliptic involution of $X_0(N)$ to involutions of 
$X_{\Delta_i}(N)$ both are bielliptic involutions. 
\par
Another criterion by Accola \cite[Corollary 1]{Ac1} says that 
a curve $X$ of genus $4$ that is an unramified degree $3$ Galois 
cover of a genus $2$ curve $Y$ will be bielliptic. More precisely, 
by \cite[Lemma 2]{Ac1} then $X$ has at least $3$ bielliptic 
involutions, namely the $3$ lifts of the hyperelliptic involution 
of $Y$. This applies to $X_{\Delta_1}(26)$, $X_{\Delta_1}(28)$, and 
$X_{\Delta_3}(37)$, all three being unramified degree $3$ covers 
of the corresponding $X_0(N)$. The latter curve is highly 
interesting, and we investigate it in detail.

\begin{Lem}\label{lem:37biell} 
$X_{\Delta_3}(37)$ is a bielliptic curve, but any bielliptic involution
must be an exceptional automorphism.
\end{Lem}

\begin{proof}
$X_{\Delta_3}(37)$ has genus $4$ and is an unramified degree $3$ Galois cover 
of the genus $2$ curve $X_0(37)$. So by \cite[Corollary 1, p. 346]{Ac1} it 
is bielliptic.
\par
The non-exceptional automorphisms of $X_{\Delta_3}(37)$ form a group $S_3$ 
generated by $[2]$ (of order $3$) and the involution $\widehat{W}_{37}$. 
The quotient by this group is the elliptic curve $X_0(37)/W_{37}$. Applying
the Hurwitz formula to this $S_3$-covering, one sees that each of the $3$ 
(conjugate) involutions has $2$ fixed points. So any bielliptic involution 
must be exceptional.
\end{proof}

\begin{Rk}
This gives us another proof that $X_{\Delta_3}(37)$ is not hyperelliptic, 
as by the Castelnuovo inequality a bielliptic curve of genus $g>3$ 
cannot be hyperelliptic.
\par
But the more important point is that this contradicts the claim in
\cite{Mo} that $X_0(37)$ is the only curve for square-free $N$ with
exceptional automorphisms.
Given the situation, of course we want to determine the complete 
automorphism group of 
$X_{\Delta_3}(37)$. For that it will be convenient to prove the preceding 
lemma with a slightly more constructive approach.
\end{Rk}

\it Alternative Proof of Lemma \ref{lem:37biell}. \rm
Let $u$ ($=[2]$) be the automorphism of $X_{\Delta_3}(37)$ of order $3$
whose quotient is $X_0(37)$. Then $Aut(X_{\Delta_3}(37))$ contains 
the group of non-exceptional automorphisms 
$B=\langle u, \widehat{W}_{37}\rangle\cong S_3$. 
\par
$X_{\Delta_3}(37)$ has genus $4$ and is an unramified Galois cover of 
the genus $2$ curve $X_0(37)$. So by \cite[Corollary 4.13]{Ac2} the 
hyperelliptic involution of $X_0(37)$ lifts to an involution $w$ on 
$X_{\Delta_3}(37)$. 
Moreover, $w$ normalizes $B$, and hence they generate a non-abelian group 
of order $12$ that has a normal $3$-Sylow subgroup whose quotient group 
is non-cyclic. So the group of order $12$ is a dihedral group $D_6$. 
\par
The center of this dihedral group is an involution $v$. So $u$, which
commutes with $v$, permutes the fixed points of $v$. If there are two 
fixed points, $u$ must fix them. But we know that $u$ has no fixed 
points. So $v$ has $6$ fixed points and it is a bielliptic involution. 
\hfill $\Box$

\begin{Thm}\label{thm:Aut37}
The full automorphism group of the curve $X_{\Delta_3}(37)$ is isomorphic
to a dihedral group $D_6$ of order $12$. The exceptional automorphisms
are two automorphisms of order $6$ and the $4$ bielliptic involutions, 
of which $3$ are conjugate in $Aut(X_{\Delta_3}(37))$ and one is central. 
The quotient by the central involution is one of the two curves that 
are $3$-isogenous to the strong Weil curve with conductor $37$ and rank $0$
($37B1$ in Cremona's table \cite{Cr}).
\end{Thm}

\begin{proof} 
In the previous elaborations we have seen that $A:=Aut(X_{\Delta_3}(37))$ 
contains a subgroup $D$, isomorphic to $D_6$, that contains $u$ ($=[2]$), 
an automorphism of order $3$ with $X_{\Delta_3}(37)/\langle u\rangle = X_0(37)$.
Now we show equality. 
\par
Let $p$ be a prime dividing $|A|$. Then $p\leq 2g+1=9$ by the Hurwitz 
formula. But $p=7$ is not possible, as then the automorphism of order 
$7$ would have exactly one fixed point, which is known to be impossible 
(see for example \cite[4.15.3, p. 41]{Ac2}). 
Also, $p=5$ is not possible, for by the same arguments as in the 
proof of Lemma \ref{lem:delta25}, $X_{\Delta_3}(37)$ is not isomorphic 
to Bring's curve. Finally, $|A|$ is not divisible by $9$. Otherwise 
the nontrivial center of the $3$-Sylow subgroup would lead to an 
automorphism of order $3$ on $X_0(37)$.
\par
As $Aut(X_0(37))$ has order $4$, $D$ is the normalizer of $u$ in $A$.
So the index of $D$ in $A$ equals the number of $3$-Sylow subgroups,
and hence must be congruent to $1$ modulo $3$. Since for genus $4$ the 
order of the automorphism group is bounded by $120$, this leaves only 
$12$ and $48$ as possible orders for $|A|$. 
\par
If $|A|=48$, the action on the four $3$-Sylow subgroups induces 
a homomorphism from $A$ into $S_4$. The kernel of this homomorphism 
must be a normal subgroup of $D$. But it cannot contain $u$; otherwise
$A$ would have a subgroup $C_3\times C_3$. So this kernel can only be 
the center of $D$, which is the bielliptic involution $v$. Thus, if 
$|A|=48$, then $A$ would induce a group of automorphisms
$A/\langle v\rangle\cong S_4$ on $X_{\Delta_3}(37)/\langle v\rangle$, 
which is known to be impossible on a genus $1$ curve.
\par
So we have shown $A\cong D_6$ and the rest of the theorem follows
easily from the structure of that group.
\end{proof}

To deal with the remaining genus $4$ cases we use

\begin{Lem}\label{fourrule}
Let $X$ be a non-hyperelliptic curve of genus $4$ that has an involution
$u$ such that $X/\langle u\rangle$ has genus $2$. If there exists an involution 
$v$ of $X$ that commutes with $u$ and induces the hyperelliptic involution on 
$X/\langle u\rangle$, then $v$ and $uv$ are bielliptic involutions of $X$.	 
\end{Lem}

\begin{proof} 
By \cite[p.49]{Ac2} we have
$$g(X/\langle u\rangle)+g(X/\langle v\rangle)+g(X/\langle uv\rangle)=g(X)+2g(X/\langle u,v\rangle)=4.$$
Since $X$ is not hyperelliptic this implies $g(X/\langle v\rangle)=g(X/\langle uv\rangle)=1$.
\end{proof}

Lemma \ref{fourrule} is much weaker than the criteria by Accola we discussed earlier in the sense that 
the hyperelliptic involution of $X/\langle u\rangle$ does not automatically lift to an involution of $X$.
It applies to $X_{\Delta_2}(26)$, $X_{\Delta_2}(28)$, $X_{\Delta_2}(29)$ and $X_{\Delta_2}(50)$. 
\par 
Actually, Accola's criteria generalize the classical result that a genus $3$ curve that has an involution 
without fixed points must be hyperelliptic \cite[Lemma 5.10]{Ac2}. This immediately implies that on 
a non-hyperelliptic curve of genus $3$ every involution (if there are) is bielliptic. So the treatment 
of the genus $3$ curves boils down to finding some involutions, for example with the help of Section
\ref{sec:autos}. For the hyperelliptic curve we have

\begin{Exm}\label{aut21}
The genus $3$ curve $X_{\Delta_1}(21)$ was already shown 
in \cite{J-K3} to be hyperelliptic. Using its equation
$$y^2 = (x^2 -x+1)(x^6 +x^5 -6x^4 -3x^3 +14x^2 -7x+1)$$
from \cite[Table 11]{B-G-G-P}, {\tt MAGMA} determines its 
automorphism group to be $D_6$; so there are no exceptional
automorphisms. The center is the hyperelliptic involution 
$\widehat{W}_3 ={9\ \ -4\choose 21\ -9}$. The remaining $6$ 
involutions come in two conjugacy classes, one without fixed points,
the other one bielliptic. Since $W_7$ has no fixed points on $X_0(21)$,
the same must hold for $\widehat{W}_7$ on $X_{\Delta_1}(21)$. So the
bielliptic involutions are $\widehat{W}_{21}$, 
$[2]\widehat{W}_{21}$, and $[4]\widehat{W}_{21}$.
\end{Exm}

So far our elaborations account for $21$ of the remaining $25$ curves. For the last 
$4$ curves, namely $X_{\Delta_2}(34)$, $X_{\Delta_3}(35)$, $X_{\Delta_4}(45)$ and 
$X_{\Delta_4}(55)$ we applied the method from Section \ref{sec:Atkin}. Compare Example 
\ref{calcfixedpoints} and Remark \ref{45and64}. 
For the full computations not treated in Example \ref{calcfixedpoints}. one can consult the web page \cite{J}.


\section{Quadratic points}\label{sec:quadratic}

In this section, we determine all the $X_\Delta(N)$ which admit infinitely 
many quadratic points over $\Q$. If $X_\Delta(N)$ is subhyperelliptic,
then it has infinitely many quadratic points because there exists
a $\Q$-rational map of degree 2 to the projective line. 
\par
Now suppose that $X_\Delta(N)$ has infinitely many quadratic points 
but is not subhyperelliptic. Then it must be bielliptic by 
\cite[Corollary 3]{H-S}. This reduces the remaining candidates 
to those listed in Theorem \ref{thm:bielliptic}. 
\par
More precisely, we have the following ``if and only if''-statement 
over a fixed based field $k$.

\begin{Thm}\label{thm:precise} \cite[Theorem 2.14]{B2} 
Let $k$ be a number field and $X$ a non-hyperelliptic curve of genus 
$g\geq 3$ over $k$ with a $k$-rational point. Then $X$ has infinitely 
many quadratic points over $k$ if and only if it has a bielliptic
involution $v$ defined over $k$ such that the elliptic curve 
$X/\langle v\rangle$ has positive rank over $k$.
\end{Thm}

Actually, except for the case $N=37$ we only need the weaker statement
that if a non-subhyperelliptic $X$ has infinitely many quadratic points
over $\Q$, then it must be bielliptic and its Jacobian variety must contain 
an elliptic curve $E$ with positive rank over $\Q$. 
\par
Now, if the Jacobian of $X_{\Delta}(N)$ or even of $X_1(N)$ contains
an elliptic curve $E$ over $\Q$, then the conductor of $E$ must divide 
$N$. From Cremona's tables \cite{Cr}, for the numbers $N$ from the
Table in Theorem \ref{thm:bielliptic} there exists no elliptic curve 
of positive rank over $\Q$ whose conductor divides $N$, except for
$N=37$. 

\begin{Lem}  $X_{\Delta_3}(37)$ has only finitely many quadratic points.
\end{Lem}

\begin{proof} Suppose $X_{\Delta_3}(37)$ has infinitely many quadratic 
points over $\Q$. Then its Jacobian variety must contain an elliptic 
curve $E$ of positive rank over $\Q$ and there must be a map of degree 
$2$ from $X_{\Delta_3}(37)$ to $E$. Now according to Stein's table 
\cite{St}, the Jacobian variety of $X_1(37)$ contains only one elliptic 
curve $E$ with positive rank over $\Q$, namely the quotient curve of 
$X_0(37)$ by $W_{37}$. There also is the natural map of degree $6$ from 
$X_{\Delta_3}(37)$ via $X_0(37)$ to $E$. Since the $\Q$-isogeny class of 
$E$ contains only one curve, a map of degree $2$ would imply the existence 
of an endomorphism of $E$ of degree $3$. But all endomorphisms of $E$ 
have degree $n^2$, since the endomorphism ring of $E$ is $\Z$. Note that 
CM-curves have integral $j$-invariant, but from Cremona's table \cite{Cr} 
the $j$-invariant of $E$ has a pole at $37$.
\end{proof}

\begin{Rk}
With practically the same proof one can see that the curve
$X_{\Delta_4}(37)$ has only finitely many quadratic points.
So to obtain that fact one doesn't need the lengthy proof 
that $X_{\Delta_4}(37)$ is not bielliptic.
\end{Rk}

Summarizing, we have proved Theorem \ref{thm:bielliptic2}.

\bigskip

\begin{center}
{\bf Acknowledgment}
\end{center}
The authors thank the referee (and several previous referees from
different journals) for detailed reports and many suggestions for 
improvement.

\section*{Appendix}

We summarize the classification of $X_0(N)$
accomplished by Ogg \cite{O2} for the hyperelliptic curves
and by Bars \cite{B1} for the bielliptic curves as follows:





\begin{center}

\begin{longtable}{l|l|c}
\caption{Classification of $X_0(N)$ to be rational, elliptic, hyperelliptic or bielliptic}\label{classification}\\

\hline  & $N$ & Number of $N$  \\ \hline
Rational & $1,\dots,10, 12, 13, 16, 18, 25$ &  15
 \\ \hline
Elliptic & $11, 14, 15, 17, 19, 20, 21, 24, 27, 32, 36, 49$ &  12
 \\ \hline
Hyperelliptic &  $22, 23, 26, 28, 29, 30, 31, 33, 35, 37, 39, 40, 41, 46,$  &  19\\
 & $47, 48, 50, 59, 71$ & 
 \\ \hline
Bielliptic & $22, 26, 28, 30, 33, 34, 35, 37, 38, 39, 40, 42, 43, 44, $ & 41\\
& $45, 48, 50, 51, 53, 54, 55, 56, 60, 61, 62, 63, 64, 65,$ & \\
& $69, 72, 75, 79, 81, 83, 89, 92, 94, 95, 101, 119, 131$ & 
 \\ \hline
\end{longtable}
\end{center}

In what follows, for each $N$ appearing in Table \ref{classification}, we list all possible $\Delta$ for which 
$g_\Delta(N)\le 1$ or no $\Delta$ in Table \ref{lowgenus} and those with $g_\Delta(N)\ge 2$ in Table \ref{highgenus} 
together with their genera, and we indicate where they are treated. 
We label the subgroup $\Delta$ depending on its order, that is, $\Delta_1, \Delta_2, \dots$ in the increasing order.

\begin{center}

\begin{longtable}{c|l|c}
\caption{List of $N$, $\Delta$ without intermediate $\Delta$ and with genera $g_\Delta(N)\le 1$}\label{lowgenus}\\

\hline $N$ & $\{\pm 1\} \subsetneq \Delta \subsetneq (\Z/N\Z)^*$ &
$g_\Delta(N)$ \\ \hline
 $1-12$ & none &  
 \\ \hline
 $13$ & $\Delta_1=\{ \pm 1, \pm 5\}$ & 0 
 \\
  & $\Delta_2=\{ \pm 1, \pm 3, \pm 4 \}$ & 0 
 \\ \hline
 $14$ & none &  
 \\ \hline
 $15$ & $\Delta_1=\{ \pm 1, \pm 4\}$ & 1 
 \\ \hline
 $16$ & $\Delta_1=\{ \pm 1, \pm 7\}$ & 0 
 \\ \hline
 $17$ & $\Delta_1=\{\pm 1,\pm 4\}$ & $1$ 
 \\
  & $\Delta_2=\{\pm 1,\pm 2,\pm 4,\pm 8\}$ & $1$ 
 \\ \hline
 $18$ & none &  
 \\ \hline
 $19$ & $\Delta_1=\{\pm 1,\pm 7,\pm 8\}$ & $1$ 
 \\ \hline
 $20$ & $\Delta_1=\{ \pm 1, \pm 9\}$ & 1 
 \\ \hline
 $21$ & $\Delta_2=\{\pm 1,\pm 4,\pm 5\}$ & $1$ 
 \\ \hline
 $22$ & none &  
 \\ \hline
 $23$ & none &  
 \\ \hline
 $24$ &  $\Delta_3=\{\pm 1,\pm 11\}$ & $1$ 
 \\ \hline
 $25$ & $\Delta_2=\{\pm 1,\pm 4,\pm 6,\pm 9,\pm 11\}$ & $0$ 
 \\ \hline
 $27$ & $\Delta_1=\{\pm 1,\pm 8,\pm 10\}$ & $1$ 
 \\ \hline
 $32$ & $\Delta_2=\{\pm 1,\pm 7,\pm 9,\pm 15\}$ & $1$ 
 \\ \hline
 $46$ & none &  
 \\ \hline
 $47$ & none &  
 \\ \hline
 $59$ & none &  
 \\ \hline
 $83$ & none &  
 \\ \hline
 $94$ & none &  
 \\ \hline
\end{longtable}
\end{center}

\begin{center}

\begin{longtable}{c|l|c|c}
\caption{List of $N$, $\Delta$ with genera $g_\Delta(N)\ge 2$}\label{highgenus}\\

\hline $N$ & $\{\pm 1\} \subsetneq \Delta \subsetneq (\Z/N\Z)^*$ &
$g_\Delta(N)$ & treated in
 \\ \hline
 $21$ & $\Delta_1=\{\pm 1,\pm 8\}$ & $3$ & Example \ref{aut21}
 \\ \hline
 $24$ & $\Delta_1=\{\pm 1,\pm 5\}$ & $3$ & Theorem \ref{thm:bielliptic}
 \\
  & $\Delta_2=\{\pm 1,\pm 7\}$ & $3$ & Theorem \ref{thm:bielliptic}
 \\ \hline
 $25$ & $\Delta_1=\{\pm 1,\pm 7\}$ & $4$ & Lemma \ref{lem:delta25}
 \\ \hline
 $26$ & $\Delta_1=\{\pm 1,\pm 5\}$ & $4$ & Theorem \ref{thm:bielliptic}
 \\
  & $\Delta_2=\{\pm 1,\pm 3,\pm 9\}$ & $4$ & Lemma \ref{fourrule}
 \\ \hline
 $28$ & $\Delta_1=\{\pm 1,\pm 13\}$ & $4$ & Theorem \ref{thm:bielliptic}
 \\
  & $\Delta_2=\{\pm 1,\pm 3,\pm 9\}$ & $4$ & Lemma \ref{fourrule}
 \\ \hline
 $29$ & $\Delta_1=\{\pm 1,\pm 12\}$ & $8$ & Lemma \ref{lem:cusp2}
 \\
  & $\Delta_2=\{\pm 1,\pm 4,\pm 5,\pm 6,\pm 7,\pm 9,\pm 13\}$ & $4$ & Lemma \ref{fourrule}
 \\ \hline
 $30$ & $\Delta_1=\{\pm 1,\pm 11\}$ & $5$ & Theorem \ref{thm:bielliptic}
 \\ \hline
 $31$ & $\Delta_1=\{\pm 1,\pm 5,\pm 6\}$ & $6$ & Lemma \ref{lem:cusp1}
 \\
  & $\Delta_2=\{\pm 1,\pm 2,\pm 4,\pm 8, \pm 15\}$ & $6$ & Lemma \ref{lem:cusp1}
 \\ \hline
 $32$ & $\Delta_1=\{\pm 1,\pm 15\}$ & $5$ & Theorem \ref{thm:bielliptic}
 \\ \hline
 $33$ & $\Delta_1=\{\pm 1,\pm 10\}$ & $11$ & Lemma \ref{lem:cusp2}
 \\
  & $\Delta_2=\{\pm 1,\pm 2,\pm 4,\pm 8, \pm 16\}$ & $5$ & Theorem \ref{thm:bielliptic}
 \\ \hline
 $34$ & $\Delta_1=\{\pm 1,\pm 13\}$ & $9$ & Lemma \ref{lem:castel-method}
 \\
  & $\Delta_2=\{\pm 1,\pm 9,\pm 13,\pm 15\}$ & $5$ & Example \ref{calcfixedpoints}
 \\ \hline
 $35$ & $\Delta_1=\{\pm 1,\pm 6\}$ & $13$ & Lemma \ref{lem:cusp2}
 \\
  & $\Delta_2=\{\pm 1,\pm 11,\pm 16\}$ & $9$ & Lemma \ref{lem:cusp2}
 \\
  & $\Delta_3=\{\pm 1,\pm 6,\pm 8,\pm 13\}$ & $7$ & Theorem \ref{thm:bielliptic}
 \\
  & $\Delta_4=\{\pm 1,\pm 4,\pm 6,\pm 9,\pm 11,\pm 16\}$ & $5$ & Theorem \ref{thm:bielliptic}
 \\ \hline
 $36$ & $\Delta_1=\{\pm 1,\pm 17\}$ & $7$ & Lemma \ref{lem:delta36}
 \\
  & $\Delta_2=\{\pm 1,\pm 11,\pm 13\}$ & $3$ & Theorem \ref{thm:bielliptic}
 \\ \hline
 $37$ & $\Delta_1=\{\pm 1,\pm 6\}$ & $16$ & Lemma \ref{lem:fp-app}
 \\
  & $\Delta_2=\{\pm 1,\pm 10,\pm 11\}$ & $10$ & Lemma \ref{lem:fp-app}
 \\
  & $\Delta_3=\{\pm 1,\pm 6,\pm 8,\pm 10,\pm 11,\pm 14\}$ & $4$ & Theorem \ref{thm:Aut37}
 \\
  & $\Delta_4=\{\pm 1,\pm 3,\pm 4,\pm 7,\pm 9,\pm 10,\pm 11,\pm 12,\pm 16\} $ & $4$ & Lemma \ref{lem:delta37}
 \\ \hline
 $38$ & $\Delta_1=\{\pm 1,\pm 7,\pm 11\}$ & $10$ & Lemma \ref{lem:cusp2}
 \\ \hline
 $39$ & $\Delta_1=\{\pm 1,\pm 14\}$ & $17$ & Corollary \ref{cor:all}
 \\
  & $\Delta_2=\{\pm 1,\pm 16,\pm 17\}$ & $9$ & Lemma \ref{lem:cusp2}
 \\
  & $\Delta_3=\{\pm 1,\pm 5,\pm 8,\pm 14\}$ & $9$ & Lemma \ref{lem:39and40}
 \\
  & $\Delta_4=\{\pm 1,\pm 4,\pm 10,\pm 14,\pm 16,\pm 17\} $ & $5$ & Theorem \ref{thm:bielliptic}
 \\ \hline
 $40$ & $\Delta_1=\{\pm 1,\pm 9\}$ & $13$ & Corollary \ref{cor:all}
 \\
  & $\Delta_2=\{\pm 1,\pm 11\}$ & $13$ & Lemma \ref{lem:fp-app}
 \\
  & $\Delta_3=\{\pm 1,\pm 19\}$ & $9$ & Lemma \ref{lem:fp-app}
 \\
  & $\Delta_4=\{\pm 1,\pm 3,\pm 9,\pm 13\}$ & $7$ & Lemma \ref{lem:39and40}
 \\
  & $\Delta_5=\{\pm 1,\pm 7,\pm 9,\pm 17\} $ & $7$ & Lemma \ref{lem:39and40}
 \\
  & $\Delta_6=\{\pm 1,\pm 9,\pm 11,\pm 19\} $ & $5$ & Theorem \ref{thm:bielliptic}
 \\ \hline
  $41$ & $\Delta_1=\{\pm 1,\pm 9\}$ & $21$ & Lemma \ref{lem:cusp2}
 \\
  & $\Delta_2=\{\pm 1,\pm 3,\pm 9,\pm 14\}$ & $11$ & Lemma \ref{lem:cusp2}
 \\
  & $\Delta_3=\{\pm 1,\pm 4,\pm 10,\pm 16,\pm 18\}$ & $11$ & Lemma \ref{lem:cusp2}
 \\
  & $\Delta_4=\{\pm 1,\pm 2,\pm 4,\pm 5,\pm 8,\pm 9,\pm 10,\pm 16,\pm 18,\pm 20\}$ & $5$ & Theorem \ref{thm:bielliptic}
 \\ \hline
 $42$ & $\Delta_1=\{\pm 1,\pm 13\}$ & $13$ & Lemma \ref{lem:cusp2}
 \\
  & $\Delta_2=\{\pm 1,\pm 5,\pm 17\}$ & $9$ & Lemma \ref{lem:cusp2}
 \\ \hline
 $43$ & $\Delta_1=\{\pm 1,\pm 6,\pm 7\}$ & $15$ & Lemma \ref{lem:cusp1}
 \\
  & $\Delta_2=\{\pm 1,\pm 2,\pm 4,\pm 8,\pm 11,\pm 16,\pm 21\}$ & $9$ & Lemma \ref{lem:cusp1}
 \\ \hline
 $44$ & $\Delta_1=\{\pm 1,\pm 21\}$ & $16$ & Lemma \ref{lem:fp-app}
 \\
  & $\Delta_2=\{\pm 1,\pm 5,\pm 7,\pm 9,\pm 19\}$ & $8$ & Lemma \ref{lem:fp-app}
 \\ \hline
  $45$ & $\Delta_1=\{\pm 1,\pm 19\}$ & $21$ & Corollary \ref{cor:all}
 \\
  & $\Delta_2=\{\pm 1,\pm 14,\pm 16\}$ & $9$ & Lemma \ref{lem:castel-method}
 \\
  & $\Delta_3=\{\pm 1,\pm 8,\pm 17,\pm 19\}$ & $11$ & Lemma \ref{lem:fp-app}
 \\
  & $\Delta_4=\{\pm 1,\pm 4,\pm 11,\pm 14,\pm 16,\pm 19\}$ & $5$ & Remark \ref{45and64}
 \\ \hline
  $48$ & $\Delta_1=\{\pm 1,\pm 7\}$ & $19$ & Corollary \ref{cor:all}
 \\
  & $\Delta_2=\{\pm 1,\pm 17\}$ & $19$ & Corollary \ref{cor:all}
 \\
  & $\Delta_3=\{\pm 1,\pm 23\}$ & $13$ & Corollary \ref{cor:all}
 \\
  & $\Delta_4=\{\pm 1,\pm 5,\pm 19,\pm 23\}$ & $7$ & Lemma \ref{lem:fp-app}
 \\
  & $\Delta_5=\{\pm 1,\pm 7,\pm 17,\pm 23\}$ & $7$ & Lemma \ref{lem:fp-app}
 \\
  & $\Delta_6=\{\pm 1,\pm 11,\pm 13,\pm 23\}$ & $5$ & Theorem \ref{thm:bielliptic}
 \\ \hline
 $49$ & $\Delta_1=\{\pm 1,\pm 18,\pm 19\}$ & $19$ & Lemma \ref{lem:castel-method}
 \\
  & $\Delta_2=\{\pm 1,\pm 6,\pm 8,\pm 13,\pm 15,\pm 20,\pm 22\}$ & $3$ & Remark \ref{extremal}
 \\ \hline
 $50$ & $\Delta_1=\{\pm 1,\pm 7\}$ & $22$ & Lemma \ref{lem:castel-method}
 \\
  & $\Delta_2=\{\pm 1,\pm 9,\pm 11,\pm 19,\pm 21\}$ & $4$ & Remark \ref{extremal}
 \\ \hline
 $51$ & $\Delta_1=\{\pm 1,\pm 16\}$ & $33$ & Lemma \ref{lem:cusp2}
 \\
  & $\Delta_2=\{\pm 1,\pm 4,\pm 13,\pm 16\}$ & $17$ & Lemma \ref{lem:cusp2}
 \\
  & $\Delta_3=\{\pm 1,\pm 2,\pm 4,\pm 8,\pm 13,\pm 16,\pm 19,\pm 25\}$ & $9$ & Lemma \ref{lem:cusp2}
 \\ \hline
 $53$ & $\Delta_1=\{\pm 1,\pm 23\}$ & $40$ & Lemma \ref{lem:cusp1}
 \\
  & $\Delta_2 =\{\pm 1,\pm 4,\pm 6,\pm 7,\pm 9, \pm 10,\pm 11,\pm 13,\pm 15,\pm 16,\pm 17,$  & $8$ & Lemma \ref{lem:cusp1} \\
  & \hspace{1.05cm} $\pm 24, \pm 25\}$ &  &
 \\ \hline
 $54$ & $\Delta_1=\{\pm 1,\pm 17,\pm 19\}$ & $10$ & Lemma \ref{lem:castel-method}
 \\ \hline
 $55$ & $\Delta_1=\{\pm 1,\pm 21\}$ & $41$ & Lemma \ref{lem:cusp2}
 \\
  & $\Delta_2=\{\pm 1,\pm 12,\pm 21,\pm 23\}$ & $21$ & Lemma \ref{lem:cusp2}
 \\
  & $\Delta_3=\{\pm 1,\pm 16,\pm 19,\pm 24,\pm 26\}$ & $17$ & Lemma \ref{lem:cusp2}
 \\
  & $\Delta_4=\{\pm 1,\pm 4,\pm 6,\pm 9,\pm 14,\pm 16,\pm 19,\pm 21,\pm 24,\pm 26\}$ & $9$ & Theorem \ref{thm:bielliptic}
 \\ \hline
 $56$ & $\Delta_1=\{\pm 1,\pm 13\}$ & $31$ & Corollary \ref{cor:all}
 \\
  & $\Delta_2=\{\pm 1,\pm 15\}$ & $31$ & Corollary \ref{cor:all}
 \\
  & $\Delta_3=\{\pm 1,\pm 27\}$ & $25$ & Corollary \ref{cor:all}
 \\
  & $\Delta_4=\{\pm 1,\pm 9,\pm 25\}$ & $21$ & Corollary \ref{cor:all}
 \\
  & $\Delta_5=\{\pm 1,\pm 13,\pm 15,\pm 27\}$ & $13$ & Lemma \ref{lem:castel-method}
 \\
  & $\Delta_6=\{\pm 1,\pm 5,\pm 9,\pm 11,\pm 13,\pm 25\}$ & $11$ & Lemma \ref{lem:fp-app}
 \\
  & $\Delta_7=\{\pm 1,\pm 9,\pm 15,\pm 17,\pm 23,\pm 25\}$ & $11$ & Lemma \ref{lem:fp-app}
 \\
  & $\Delta_8=\{\pm 1,\pm 3,\pm 9,\pm 19,\pm 25,\pm 27\}$ & $9$ & Lemma \ref{lem:elliptic}
 \\ \hline
 $60$ & $\Delta_1=\{\pm 1,\pm 11\}$ & $29$ & Lemma \ref{lem:cusp2}
 \\
  & $\Delta_2=\{\pm 1,\pm 19\}$ & $29$ & Lemma \ref{lem:cusp2}
 \\
  & $\Delta_3=\{\pm 1,\pm 29\}$ & $25$ & Lemma \ref{lem:cusp2}
 \\
  & $\Delta_4=\{\pm 1,\pm 7,\pm 11,\pm 17\}$ & $15$ & Lemma \ref{lem:cusp2}
 \\
  & $\Delta_5=\{\pm 1,\pm 11,\pm 13,\pm 23\}$ & $15$ & Lemma \ref{lem:cusp2}
 \\
  & $\Delta_6=\{\pm 1,\pm 11,\pm 19,\pm 29\}$ & $13$ & Lemma \ref{lem:cusp2}
 \\ \hline
 $61$ & $\Delta_1=\{\pm 1,\pm 11\}$ & $56$ & Lemma \ref{lem:cusp1}
 \\
  & $\Delta_2=\{\pm 1,\pm 13,\pm 14\}$ & $36$ & Lemma \ref{lem:cusp1}
 \\
  & $\Delta_3=\{\pm 1,\pm 3, \pm 9,\pm 20,\pm 27\}$ & $26$ & Lemma \ref{lem:cusp1}
 \\
  & $\Delta_4=\{\pm 1,\pm 11,\pm 13,\pm 14, \pm 21,\pm 29\}$ & $16$ & Lemma \ref{lem:cusp1}
 \\
  & $\Delta_5=\{\pm 1,\pm 3,\pm 8,\pm 9, \pm 11,\pm 20,\pm 23,\pm 24, \pm 27,\pm 28\}$ & $12$ & Lemma \ref{lem:cusp1}
 \\
  & $\Delta_6=\{\pm 1,\pm 3,\pm 4,\pm 5, \pm 9,\pm 12,\pm 13,\pm 14, \pm 15,\pm 16,\pm 19,$ & $8$ & Lemma \ref{lem:cusp1}
 \\
  & $\hspace{1.2cm}\pm 20,\pm 22,\pm 25,\pm 27\}$ &   &  
 \\ \hline
 $62$ & $\Delta_1=\{\pm 1,\pm 5,\pm 25\}$ & $31$ & Lemma \ref{lem:cusp2}
 \\
  & $\Delta_2=\{\pm 1,\pm 15,\pm 23,\pm 27,\pm 29\}$ & $19$ & Lemma \ref{lem:cusp2}
 \\ \hline
 $63$ & $\Delta_1=\{\pm 1,\pm 8\}$ & $49$ & Corollary \ref{cor:all}
 \\
  & $\Delta_2=\{\pm 1,\pm 4,\pm 16\}$ & $33$ & Corollary \ref{cor:all}
 \\
  & $\Delta_3=\{\pm 1,\pm 5,\pm 25\}$ & $33$ & Corollary \ref{cor:all}
 \\
  & $\Delta_4=\{\pm 1,\pm 17,\pm 26\}$ & $33$ & Corollary \ref{cor:all}
 \\
  & $\Delta_5=\{\pm 1,\pm 20,\pm 22\}$ & $25$ & Corollary \ref{cor:all}
 \\
  & $\Delta_6=\{\pm 1,\pm 2,\pm 4,\pm 8,\pm 16,\pm 31\}$ & $17$ & Lemma \ref{lem:fp-app}
 \\
  & $\Delta_7=\{\pm 1,\pm 5,\pm 8,\pm 11,\pm 23,\pm 25\}$ & $17$ & Lemma \ref{lem:fp-app}
 \\
  & $\Delta_8=\{\pm 1,\pm 8,\pm 10,\pm 17,\pm 19,\pm 26\}$ & $17$ & Lemma \ref{lem:fp-app}
 \\
  & $\Delta_9=\{\pm 1,\pm 8,\pm 13,\pm 20,\pm 22,\pm 29\}$ & $13$ & Lemma \ref{lem:elliptic}
 \\
  & $\Delta_{10}=\{\pm 1,\pm 4,\pm 5,\pm 16,\pm 17,\pm 20,\pm 22,\pm 25,\pm 26\}$ & $9$ & Lemma \ref{lem:castel-method}
 \\ \hline
 $64$ & $\Delta_1=\{\pm 1,\pm 31\}$ & $37$ & Corollary \ref{cor:all}
 \\
  & $\Delta_2=\{\pm 1,\pm 15,\pm 17,\pm 31\}$ & $13$ & Lemma \ref{lem:fp-app}
 \\
  & $\Delta_3=\{\pm 1,\pm 7,\pm 9,\pm 15,\pm 17,\pm 23,\pm 25,\pm 31\}$ & $5$ & Remark \ref{extremal}
 \\ \hline
$65$ & $\Delta_1=\{\pm 1,\pm 8\}$ & $55$ & Lemma \ref{lem:cusp1}
 \\
  & $\Delta_2=\{\pm 1,\pm 14\}$ & $61$ & Lemma \ref{lem:cusp1}
 \\
  & $\Delta_3=\{\pm 1,\pm 18\}$ & $55$ & Lemma \ref{lem:cusp1}
 \\
  & $\Delta_4=\{\pm 1,\pm 4,\pm 16\}$ & $41$ & Lemma \ref{lem:cusp1}
 \\
  & $\Delta_5=\{\pm 1,\pm 8,\pm 14,\pm 18\}$ & $25$ & Lemma \ref{lem:cusp1}
 \\
  & $\Delta_6=\{\pm 1,\pm 12,\pm 14,\pm 27\}$ & $31$ & Lemma \ref{lem:cusp1}
 \\
  & $\Delta_7=\{\pm 1,\pm 14, \pm 21,\pm 31\}$ & $31$ & Lemma \ref{lem:cusp1}
 \\
  & $\Delta_8=\{\pm 1,\pm 2,\pm 4,\pm 8, \pm 16,\pm 32\}$ & $19$ & Lemma \ref{lem:cusp1}
 \\
  & $\Delta_9=\{\pm 1,\pm 4,\pm 7,\pm 16, \pm 18,\pm 28\}$ & $19$ & Lemma \ref{lem:cusp1}
 \\
  & $\Delta_{10}=\{\pm 1,\pm 4,\pm 9,\pm 14, \pm 16,\pm 29\}$ & $21$ & Lemma \ref{lem:cusp1}
 \\
  & $\Delta_{11}=\{\pm 1,\pm 8, \pm 12,\pm 14,\pm 18,\pm 21,\pm 27, \pm 31\}$ & $13$ & Lemma \ref{lem:cusp1}
 \\
   & $\Delta_{12}=\{\pm 1,\pm 2,\pm 4,\pm 7, \pm 8,\pm 9,\pm 14,\pm 16, \pm 18,\pm 28,\pm 29,$ & $9$ & Lemma \ref{lem:cusp1}
 \\
  & $\hspace{1.3cm}\pm 32\}$ &   &  
 \\
  & $\Delta_{13}=\{\pm 1,\pm 4,\pm 6, \pm 9,\pm 11, \pm 14,\pm 16,\pm 19,\pm 21,\pm 24,\pm 29,$ & $11$ & Lemma \ref{lem:cusp1}
 \\
  & $\hspace{1.3cm}\pm 31\}$ &   &  
 \\  
 & $\Delta_{14}=\{\pm 1,\pm 3,\pm 4,\pm 9, \pm 12,\pm 14,\pm 16,\pm 17, \pm 22,\pm 23,\pm 27,$ & $11$ & Lemma \ref{lem:cusp1}
 \\
  & $\hspace{1.3cm}\pm 29\}$ &   &  
 \\
 \hline
 $69$ & $\Delta_1=\{\pm 1,\pm 22\}$ & $67$ & Lemma \ref{lem:cusp2}
 \\
  & $\Delta_2=\{\pm 1,\pm 4,\pm 5,\pm 11,\pm 13,\pm 14,\pm 16,\pm 17,\pm 20,\pm 25,\pm 31\}$ & $13$ & Lemma \ref{lem:elliptic}
 \\ \hline
 $71$ & $\Delta_1=\{\pm 1,\pm 5,\pm 14,\pm 17,\pm 25\}$ & $36$ & Lemma \ref{lem:cusp1}
 \\
  & $\Delta_2=\{\pm 1,\pm 20,\pm 23,\pm 26,\pm 30,\pm 32,\pm 34\}$ & $26$ & Lemma \ref{lem:cusp1}
 \\ \hline
 $72$ & $\Delta_1=\{\pm 1,\pm 17\}$ & $49$ & Corollary \ref{cor:all}
 \\
  & $\Delta_2=\{\pm 1,\pm 19\}$ & $49$ & Corollary \ref{cor:all}
 \\
  & $\Delta_3=\{\pm 1,\pm 35\}$ & $41$ & Corollary \ref{cor:all}
 \\
  & $\Delta_4=\{\pm 1,\pm 23,\pm 25\}$ & $25$ & Corollary \ref{cor:all}
 \\
  & $\Delta_5=\{\pm 1,\pm 17,\pm 19\, \pm 35\}$ & $21$ & Lemma \ref{lem:castel-method}
 \\
  & $\Delta_6=\{\pm 1,\pm 5,\pm 19,\pm 23,\pm 25,\pm 29\}$ & $13$ & Lemma \ref{lem:fp-app}
 \\
  & $\Delta_7=\{\pm 1,\pm 7,\pm 17,\pm 23,\pm 25,\pm 31\}$ & $13$ & Lemma \ref{lem:fp-app}
 \\
  & $\Delta_8=\{\pm 1,\pm 11,\pm 13,\pm 23,\pm 25,\pm 35\}$ & $9$ & Lemma \ref{lem:castel-method}
 \\ \hline
 $75$ & $\Delta_1=\{\pm 1,\pm 26\}$ & $73$ & Lemma \ref{lem:cusp1}
 \\
  & $\Delta_2=\{\pm 1,\pm 7,\pm 26,\pm 32\}$ & $37$ & Lemma \ref{lem:cusp1}
 \\
  & $\Delta_3=\{\pm 1,\pm 14,\pm 16,\pm 29,\pm 31\}$ & $14$ & Lemma \ref{lem:cusp1}
 \\
  & $\Delta_4=\{\pm 1,\pm 4,\pm 11,\pm 14,\pm 16,\pm 19,\pm 26,\pm 29,\pm 31,\pm 34\}$ & $9$ & Lemma \ref{lem:cusp1}
 \\ \hline
 $79$ & $\Delta_{1}=\{\pm 1,\pm 23,\pm 24\}$ & $66$  & Lemma \ref{lem:cusp1}
 \\
   & $\Delta_{2}=\{\pm 1,\pm 8,\pm 10,\pm 12,\pm 14,\pm 15,\pm 17,\pm 18,\pm 21,\pm 22,\pm 27,$ & $18$ & Lemma \ref{lem:cusp1}
 \\
  & $\hspace{1.2cm}\pm 33,\pm 38\}$ &   &  
 \\  \hline
  $81$ & $\Delta_1=\{\pm 1,\pm 26,\pm 28\}$ & $46$ & Corollary \ref{cor:all}
 \\
  & $\Delta_2=\{\pm 1,\pm 8,\pm 10,\pm 17,\pm 19,\pm 26,\pm 28,\pm 35,\pm 37\}$ & $10$ & Lemma \ref{lem:castel-method}
 \\ \hline
 $89$ & $\Delta_1=\{\pm 1,\pm 34\}$ & $133$ & Lemma \ref{lem:cusp1}
 \\
   & $\Delta_2=\{\pm 1,\pm 12,\pm 34,\pm 37\}$ & $67$ & Lemma \ref{lem:cusp1}
 \\
   & $\Delta_3=\{\pm 1,\pm 2,\pm 4,\pm 8,\pm 11,\pm 16,\pm 22,\pm 25,\pm 32,\pm 39,\pm 44\}$ & $27$ & Lemma \ref{lem:cusp1}
 \\
   & $\Delta_4=\{\pm 1,\pm 2,\pm 4,\pm 5,\pm 8,\pm 9,\pm 10,\pm 11,\pm 16,\pm 17,\pm 18,\pm 20,$ & $13$ & Lemma \ref{lem:cusp1}
 \\
  & $\hspace{1.2cm}\pm 21,\pm 22,\pm 25,\pm 32,\pm 34,\pm 36,\pm 39,\pm 40,\pm 42,\pm 44\}$ &   &  
 \\ \hline
  $92$ & $\Delta_1=\{\pm 1,\pm 45\}$ & $100$ & Lemma \ref{lem:cusp2}
 \\
  & $\Delta_2=\{\pm 1,\pm 7,\pm 9,\pm 11,\pm 13,\pm 15,\pm 19,\pm 25,\pm 29,\pm 41,\pm 43\}$ & $20$ & Lemma \ref{lem:cusp2}
 \\ \hline
$95$ & $\Delta_1=\{\pm 1,\pm 39\}$ & $145$ & Lemma \ref{lem:cusp1}
 \\
  & $\Delta_2=\{\pm 1,\pm 11,\pm 26\}$ & $97$ & Lemma \ref{lem:cusp1}
 \\
  & $\Delta_3=\{\pm 1,\pm 18,\pm 37,\pm 39\}$ & $73$  & Lemma \ref{lem:cusp1}
 \\
  & $\Delta_4=\{\pm 1,\pm 11,\pm 26,\pm 31,\pm 39,\pm 46\}$ & $49$  & Lemma \ref{lem:cusp1}
 \\
  & $\Delta_5=\{\pm 1,\pm 6,\pm 11,\pm 14,\pm 16,\pm 26,\pm 29,\pm 34,\pm 36\}$ & $33$ & Lemma \ref{lem:cusp1}
 \\
  & $\Delta_6=\{\pm 1,\pm 7,\pm 8,\pm 11,\pm 12,\pm 18,\pm 26,\pm 27,\pm 31,\pm 37,\pm 39,$ & $25$ & Lemma \ref{lem:cusp1}
 \\
  & $\hspace{1.2cm}\pm 46\}$ &   &  
 \\
  & $\Delta_{7}=\{\pm 1,\pm 4,\pm 6,\pm 9,\pm 11,\pm 14,\pm 16,\pm 21,\pm 24,\pm 26,\pm 29,$ & $17$ & Lemma \ref{lem:cusp1}
 \\
  & $\hspace{1.2cm}\pm 31,\pm 34,\pm 36,\pm 39,\pm 41,\pm 44,\pm 46\}$ &   &  
 \\ \hline
$101$ & $\Delta_1=\{\pm 1,\pm 10\}$ & $176$ & Lemma \ref{lem:cusp1}
 \\
  & $\Delta_2=\{\pm 1,\pm 6, \pm 14,\pm 17,\pm 36\}$ & $76$ & Lemma \ref{lem:cusp1}
 \\
  & $\Delta_3=\{\pm 1,\pm 6,\pm 10,\pm 14,\pm 17,\pm 32,\pm 36,\pm 39,\pm 41,\pm 44\}$ & $36$  & Lemma \ref{lem:cusp1}
 \\
  & $\Delta_4=\{\pm 1,\pm 4,\pm 5,\pm 6,\pm 9,\pm 13,\pm 14,\pm 16,\pm 17,\pm 19,\pm 20,$ & $16$ & Lemma \ref{lem:cusp1}
 \\
  & $\hspace{1.2cm}\pm 21,\pm 22,\pm 23,\pm 24,\pm 25,\pm 30,\pm 31,\pm 33,\pm 36,\pm 37,$ &   &  
 \\
  & $\hspace{1.2cm}\pm 43,\pm 45,\pm 47,\pm 49\}$ &   &  
 \\  \hline
$119$ & $\Delta_1=\{\pm 1,\pm 50\}$ & $241$ & Lemma \ref{lem:cusp1}
 \\
  & $\Delta_2=\{\pm 1,\pm 18,\pm 33\}$ & $161$ & Lemma \ref{lem:cusp1}
 \\
  & $\Delta_3=\{\pm 1,\pm 13,\pm 50,\pm 55\}$ & $121$  & Lemma \ref{lem:cusp1}
 \\
  & $\Delta_4=\{\pm 1,\pm 16,\pm 18,\pm 33,\pm 50,\pm 52\}$ & $81$  & Lemma \ref{lem:cusp1}
 \\
  & $\Delta_5=\{\pm 1,\pm 8,\pm 13,\pm 15,\pm 36,\pm 43,\pm 55,\pm 69\}$ & $61$ & Lemma \ref{lem:cusp1}
 \\
  & $\Delta_6=\{\pm 1,\pm 4,\pm 13,\pm 16,\pm 18,\pm 30,\pm 33,\pm 38,\pm 47,\pm 50,\pm 52,$ & $41$ & Lemma \ref{lem:cusp1}
 \\
  & $\hspace{1.2cm}\pm 55\}$ &   &  
 \\
  & $\Delta_{7}=\{\pm 1,\pm 6,\pm 8,\pm 13,\pm 15,\pm 20,\pm 22,\pm 27,\pm 29,\pm 41,\pm 43,$ & $31$ & Lemma \ref{lem:cusp1}
 \\
  & $\hspace{1.2cm}\pm 48,\pm 50,\pm 55,\pm 57\}$ &   &  
 \\
  & $\Delta_{8}=\{\pm 1,\pm 2,\pm 4,\pm 8,\pm 9,\pm 13,\pm 15,\pm 16,\pm 18,\pm 19,\pm 25,$ & $21$ & Lemma \ref{lem:cusp1}
 \\
  & $\hspace{1.2cm}\pm 26,\pm 30,\pm 32,\pm 33,\pm 36,\pm 38,\pm 43,\pm 47,\pm 50,\pm 52,$ &   &  
   \\
  & $\hspace{1.2cm}\pm 53,\pm 55,\pm 59\}$ &   &  
 \\ \hline
$131$ & $\Delta_1=\{\pm 1,\pm 42,\pm 53,\pm 58,\pm 61\}$ & $131$ & Lemma \ref{lem:cusp1}
 \\
  & $\Delta_2=\{\pm 1,\pm 18,\pm 19,\pm 24,\pm 32,\pm 39,\pm 45,\pm 47,\pm 51,\pm 52$ & $51$ & Lemma \ref{lem:cusp1}
 \\
  & $\hspace{1.2cm}\pm 60,\pm 62,\pm 68\}$ &   &  
 \\ \hline

\end{longtable}
\end{center}

\end{document}